\documentclass[final,3p,times,12pt]{elsarticle}

\usepackage{amsmath,amssymb,mathtools,amsthm}  
\usepackage{graphicx}                           
\usepackage{caption,subcaption}                
\usepackage{enumitem}                           
\usepackage{xcolor}                             
\usepackage{mathrsfs}    
\usepackage{multirow}%
\usepackage{array}
\usepackage{hhline}    
\usepackage{hyperref}             
\usepackage{ulem}
\usepackage{float}

\definecolor{brown}{rgb}{0.6,0.3,0}   
\definecolor{blue}{rgb}{0,0,1}

\newtheorem{lemma}{Lemma}[section]
\newtheorem{theorem}{Theorem}[section]
\newtheorem{proposition}{Proposition}[section]
\newtheorem{corollary}{Corollary}[section]
\newtheorem{remark}{Remark}[section]
\newtheorem{example}{Example}[section]

\numberwithin{equation}{section}  

\newcommand{\beq}{\begin{equation}}
	\newcommand{\eeq}{\end{equation}}

\usepackage[numbers]{natbib}
\usepackage{booktabs}

\journal{ }
\begin{document}
	\begin{frontmatter}
		\title{Two-point completion of zero interlacing and Wronskian-type bounds, with applications to Jacobi and other orthogonal polynomials}

		\author[label1]{Kerstin Jordaan}\ead{jordakh@unisa.ac.za}
		\author[label2]{Vikash Kumar}{}
		\ead{kumarv@unisa.ac.za}
		\ead{vikaskr0006@gmail.com}
		
		\cortext[cor1]{Corresponding author}
		
		\address[label1]{College of Economic and Management Sciences, University of South Africa, Pretoria, 0003, South Africa}
		\address[label2]{Department of Decision Sciences, University of South Africa, Pretoria, 0003, South Africa}
		
		\begin{abstract}
		 We investigate completed interlacing of zeros for pairs of polynomial sequences that fail to interlace by exactly two points. Using suitable mixed recurrence relations, we identify two additional points, called the completion points, required to obtain completed interlacing and establish general results for polynomials with real zeros. We also show that completed interlacing yields Wronskian-type bounds for the original polynomial pairs. The completed-interlacing results are applied to Jacobi, Meixner-Pollaczek, and Pseudo-Jacobi polynomials. In the Jacobi case, we improve earlier results by determining two explicit additional points that complete the interlacing of $P_n^{(\alpha,\beta)}$ and $P_{n+1}^{(\alpha+1,\beta+1)}$. We also analyze the locations of these points and show that the common-zero cases are non-generic in the Jacobi parameter domain. For this doubly-shifted Jacobi pair, the Wronskian-type bounds yield cross-family Turán-type inequalities and lead to an associated Stieltjes representation and complete-monotonicity consequences. For Meixner-Pollaczek polynomials, we address an open question concerning consecutive degrees with the parameter increased by one, and we obtain analogous completed interlacing results for Pseudo-Jacobi polynomials. 
		\end{abstract}
		
		
		\end{frontmatter}

	\section{Introduction}

Let $\mu$ be a finite positive Borel measure on the real line $\mathbb{R}$ whose support, denoted by $\operatorname{supp}(\mu)$, contains infinitely many points, and let $\mathbb{N}_0=\{0,1,2,\dots\}$. Define
   \[\mathcal{I}_\mu:=\left\{n\in\mathbb{N}_0 :\int_{\operatorname{supp}(\mu)} x^{2n}\,d\mu(x)<\infty\right\}.\] Since $\mu$ is finite, $\mathcal{I}_\mu$ is either $\mathbb{N}_0$ or a  set of the form $\{0,1,\dots,N\}$. For each $n\in \mathcal{I_\mu}$, let $p_n(x) = x^n + \dots$ be a monic polynomial of exact degree $n$. We say that $\{p_n(x)\}_{n \in \mathcal{I_\mu}}$ forms an orthogonal polynomial sequence with respect to $\mu$ if
	\begin{align}
		\int_{supp(\mu)} p_n(x)p_m(x) d\mu(x) = \kappa_n \delta_{n,m}, \quad \text{for all } n, m \in \mathcal{I_\mu},
	\end{align}
	where $\kappa_n > 0$.  For instance, the index set $\mathcal{I}_{\mu}$ for Jacobi and Meixner-Pollaczek polynomials is $\mathbb{N}\cup \{0\}$, whereas for Pseudo-Jacobi polynomials, the index set is finite \cite{Koekoek-Book-HypergeometricOP}. 
    
    The well-known properties of the orthogonality relation are that for each $n \in \mathcal{I}_{\mu}$, the zeros of orthogonal polynomials $p_n(x)$ are all real, distinct, and simple. These zeros lie in the interior of the convex hull of the support of the measure $\mu$ \cite{Ismail-Book-OP-2009}. The orthogonality also ensures that the zeros of consecutive degree polynomials satisfy a specific pattern known as the interlacing property (or separation theorem) \cite{Chiharabook}.
    
	The zeros of any two polynomials $P_n$ and $Q_{n-1}$ interlace, denoted $P_n \prec Q_{n-1}$, if
	\begin{align}
		z_{1,n} < x_{1,n-1} < z_{2,n} < x_{2,n-1} < \cdots < x_{n-1,n-1} < z_{n,n},
	\end{align}
	where $\{z_{k,n}\}_{k=1}^n$ and $\{x_{k,n-1}\}_{k=1}^{n-1}$ denote the zeros of $P_n$ and $Q_{n-1}$  respectively in ascending order, while we say that the zeros of $P_n$ alternate the zeros of $Q_{n}$, denoted $P_n \prec Q_n$, if
	\begin{align}
		z_{1,n} < x_{1,n} < z_{2,n} <x_{2,n}< \cdots < x_{n-1,n} < z_{n,n} < x_{n,n}.
	\end{align}
	
	The computational relevance of such zero configurations comes in part from Gaussian quadrature, since the zeros of orthogonal polynomials are the quadrature nodes, and they are commonly computed through the eigenvalues of the associated Jacobi matrix (cf. \cite{GolubWelsch1969,Gautschi2004Book}). Interlacing is therefore useful in the analysis and comparison of quadrature nodes arising from related polynomial families (cf. \cite{DriverJordaanMbuyi2009}). 
	
	The interlacing property of two orthogonal polynomials, $p_n$ and $p_m$, from the same family (where $m - n > 1$) was first studied by Stieltjes. He showed  (cf. \cite[page 46]{Szego-OP-Book}) that for $m > n + 1$, at least one zero of $p_m$ lies between any two consecutive zeros of $p_n$. This result is commonly referred to as Stieltjes interlacing.  This was further extended  in \cite{Beardon-Extension-Stieljtes-CMFT-2011,Boor-Saff-StieltjesExtension-LAA-1986} by identifying a polynomial $d_{m-n-1}$ of degree $m-n-1$ that helps to complete the interlacing property as $d_{m-n-1} p_n \prec p_m$, provided that $p_n$ and $p_m$ do not have a common zero.  The polynomials $d_{m-n-1}$ were initially introduced as dual polynomials by de Boor and Saff \cite{Boor-Saff-StieltjesExtension-LAA-1986} and subsequently identified as the associated polynomials of  $p_m$ in the work of Vinet and Zhedanov \cite{Vinet-Zhedanov-JCAM-2004}. Beyond the case of a single sequence of orthogonal polynomials, a polynomial $d_k$ of degree $k$ can also complete the interlacing of zeros between two different systems of orthogonal polynomials. For instance, in \cite{Driver-Jordaan-IndagMath-2011} a first-degree polynomial $d_{1,t}$ was identified such that $d_{1,t} L_{n-1}^{(\alpha+t)} \prec L_{n+1}^{(\alpha)}$ for $t \in \{0, 1, 2, 3, 4\}$, where $L_n^{(\alpha)}$ denotes the Laguerre polynomial of degree $n$ with parameter $\alpha > -1$. Similarly, \cite{Driver-NumerMath-2012} identified a second-degree polynomial $d_{2,t}$ such that $d_{2,t} C_{n-2}^{(\lambda+t)} \prec C_{n+1}^{(\lambda)}$ for $t \in \{0, 1, 2, 3\}$, where $C_n^{(\lambda)}$ represents the ultraspherical (Gegenbauer) polynomial of degree $n$ and parameter $\lambda > -1/2$. 
	
	It is possible to relax the constraint of orthogonality to identify a polynomial $d_l$ of degree $l$ that completes the interlacing of zeros between two arbitrary polynomials from different sequences, (cf.  \cite{Driver-Jordaan-JAT-2012,Jooste-Jordaan-Numeralg-2025,Jordaan-Kumar-onepoint-2026}).   
	
	Let $P_i$, $G_j$, and $Q_k$ denote polynomials of degrees $i$, $j$, and $k$, respectively, whose zeros are real and lie in an interval $(a,b)$. Assume that $ Q_k\prec G_j $ whenever $k=j+1$ or $G_j \prec Q_k$ whenever $k=j-1$. These polynomials satisfy the general mixed recurrence relation
	\begin{align}\label{MoreGeneralMixedTTRR}
		A(x)P_i(x)=B(x)G_j(x)+d_l(x)Q_k(x),
	\end{align}
	where $A(x)$, $B(x)$, and $d_l(x)$ are continuous functions on $(a,b)$.  
	
	In \cite[Theorem 2.1]{Jooste-Jordaan-Numeralg-2025}, the mixed recurrence relation \eqref{MoreGeneralMixedTTRR} is applied with $i=n, j=n, k=n+1$, and $d_1(x)=x-E$, to complete the interlacing of zeros of  $P_n(x)$ and $G_n(x)$. The result is then applied to Meixner-Pollaczek, Pseudo-Jacobi and continuous Hahn polynomials. Extending this framework, \cite{Jordaan-Kumar-onepoint-2026} develops various general results for completing interlacing by identifying the degree-one polynomials. These general results are applied to several orthogonal families (Laguerre, Jacobi, Krawtchouk and Meixner polynomials), as well as a class of non-orthogonal polynomials (Narayana polynomials), to establish new completed interlacing results using degree-one polynomials.
	
	Motivated by these results, a natural question arises regarding completed interlacing of two distinct polynomial sequences that fail to fully interlace by exactly two points. Specifically, we aim to explicitly identify two additional points, called the completion points, which are given by the zeros of a quadratic factor and are necessary for completing the interlacing. 
	
To achieve this, we study the general framework based on mixed recurrence relations with $i=n$, $j=n+1$, $k=n$, and $d_2(x)=\mp(x-E_1)(x-E_2)$. These relations isolate a quadratic factor whose zeros $E_1$ and $E_2$ are the completion points.  In the Jacobi case, the exact location of these points is determined. 	We also obtain an exact criterion for common zeros of
			$P_n^{(\alpha,\beta)}$ and $P_{n+1}^{(\alpha+1,\beta+1)}$.  We further show
			that, for each fixed degree, the set of parameters giving rise to common zeros
			of this shifted Jacobi pair is negligible both topologically and
			measure-theoretically.  Thus common zeros are non-generic in the parameter
			domain.

		The general completed interlacing result also yields two  Wronskian-type bounds for the original pair of polynomials.  When specialized to the Jacobi case, these Wronskian-type bounds take the form of
			new cross-family Tur\'an-type inequalities for the doubly shifted pair
			$P_n^{(\alpha,\beta)}$ and $P_{n+1}^{(\alpha+1,\beta+1)}$.  These inequalities
			are in the spirit of classical Tur\'an-type inequalities and their
			parameter-shifted variants
			\cite{ BustozJMMA1981,Gasper1972-PAMS-Turaninequality-Jacobipolynomials, Szego1948-Turaninequality};
			however, the inequalities obtained here are derived from the two-point completed interlacing mechanism. The associated quotient involving Jacobi polynomials also leads to a Stieltjes-transform representation and complete monotonicity.

        The general mixed recurrence relation approach described previously also enables us to refine existing results and provide partial answers to open
		problems in the literature. First, in the Jacobi case, we improve upon
		 \cite[Theorem~2]{Arvesu-Driver-Littlejohn-RamanujanJ-2023} by establishing the completed interlacing of
		$P_n^{(\alpha,\beta)}$ and $P_{n+1}^{(\alpha+1,\beta+1)}$ using two explicitly
		determined completion points. Second, for Meixner-Pollaczek polynomials, we address the open question posed in
		\cite{Jooste-Jordaan-Numeralg-2025} by identifying
		conditions under which the zeros of $P_n^{(\lambda)}(x;\phi)$ and
		$P_{n+1}^{(\lambda+1)}(x;\phi)$ interlace.  Finally, we establish analogous
		completed interlacing results for Pseudo-Jacobi polynomials
		$P_n(x;a,b)$ and $P_{n+1}(x;a+1,b)$.
	
	The manuscript is organized as follows: In Section \ref{General interlacing Results}, we prove general results for completing the interlacing of arbitrary polynomials $P_n$ and $G_{n+1}$ by identifying the zeros of a quadratic factor $d_2$ as the two completion points.  Section~\ref{CI and Wronskian-Turan} uses completed interlacing to derive general Wronskian-type bounds. Section \ref{sec:Jacobianalysis and Inequalities} specializes these results to Jacobi polynomials: Section \ref{Jacobi polynomials} establishes completed-interlacing results, analyzes the location of the completion points and characterizes the common-zero problem, while Sections \ref{Turan inequalities} and \ref{Stieltjes} derive cross-family Tur\'an-type inequalities, associated Stieltjes representations, and complete monotonicity properties. Finally, Sections \ref{Meixner-Pollaczek polynomials} and \ref{Pseudo Jacobi polynomials} present family-specific specializations for Meixner-Pollaczek and Pseudo-Jacobi polynomials, respectively.

\section{General two-point completed interlacing results} \label{General interlacing Results}

We first establish a parity-counting lemma that will be used in the proofs of both main results on two-point interlacing completion. Let \[ y_{1,n+1}<y_{2,n+1}<\cdots<y_{n+1,n+1} \] be the zeros of $G_{n+1}$, and define the intervals 
\[ I_0=(-\infty,y_{1,n+1}),\qquad I_k=(y_{k,n+1},y_{k+1,n+1}),\quad 1\leq k\leq n, \quad\text{ and }\quad I_{n+1}=(y_{n+1,n+1},\infty). \] Let \[ F_{n+2}(x)=(x-E_1)(x-E_2)P_n(x). \] 

\begin{lemma}\label{lem:parity-counting} Let $P_n$, $G_{n+1}$, and $Q_n$ be monic polynomials of degrees $n$, $n+1$, and $n$, respectively, whose zeros are real and simple. Assume that $G_{n+1}\prec Q_n$ and that, for all real $x$, \begin{equation} A(x)P_n(x) = B(x)G_{n+1}(x) + \varepsilon (x-E_1)(x-E_2)Q_n(x), \qquad \varepsilon\in\{-1,1\}, \label{eq:general-signed-recurrence} \end{equation} where $A(x)>0$, $E_1,E_2\in \mathbb{R}$, $E_1<E_2$. Suppose further that $P_n$ and $G_{n+1}$ are coprime and that $B(E_i)\neq 0$, $i=1,2.$ Then $F_{n+2}$ has an odd number of zeros in every interior interval $I_k$, $k=1,\dots,n$. Moreover, the number of zeros of $F_{n+2}$ in each exterior interval $I_0$ and $I_{n+1}$ is odd when $\varepsilon=-1$ and even when $\varepsilon =1$.\end{lemma} 

\begin{proof} Since $P_n$ and $G_{n+1}$ are coprime, $A(E_i)>0$, and $B(E_i)\neq0$, it follows from \eqref{eq:general-signed-recurrence} that neither $E_1$ nor $E_2$ is a zero of $P_n$ or $G_{n+1}$. Indeed, if $P_n(E_i)$=0, then \eqref{eq:general-signed-recurrence} and $B(E_i)\neq 0$ imply $G_{n+1}(E_i)=0$, which contradicts the coprimality assumption. Conversely, if $G_{n+1}(E_i)=0$, then $A(E_i)>0$ implies $P_{n}(E_i)=0$, which is again the same contradiction. Therefore, none of the zeros $y_{j,n+1}$, $j=1,\dots,n+1$ of $G_{n+1}$ is a zero of $F_{n+2}$, and the zeros $E_1$ and $E_2$ are distinct from the zeros of $P_n$. Let $y_{j,n+1}$ be any zero of $G_{n+1}$. Evaluating \eqref{eq:general-signed-recurrence} at $y_{j,n+1}$ gives
\[ A(y_{j,n+1})  P_n(y_{j,n+1}) = \varepsilon(y_{j,n+1}-E_1)(y_{j,n+1}-E_2)Q_n(y_{j,n+1}).\] Multiplying by $(y_{j,n+1}-E_1)(y_{j,n+1}-E_2)$ yields \begin{align} \label{eq:endpoint-product-parity}F_{n+2}(y_{j,n+1})=\varepsilon \frac{(y_{j,n+1}-E_1)^2(y_{j,n+1}-E_2)^2}{A(y_{j,n+1})}Q_n(y_{j,n+1}).\end{align} The coefficient of $\varepsilon Q_n(y_{j,n+1})$ is strictly positive, hence \begin{align}\label{eq:parity}\operatorname{sign}\!\left(F_{n+2}(y_{j,n+1})\right)=\operatorname{sign}\!\left(\varepsilon Q_n(y_{j,n+1})\right).\end{align} Now fix $k\in\{1,\ldots,n\}$. Since the zeros of $G_{n+1}$ and $Q_n$ strictly interlace, $Q_n$ has exactly one zero in $I_k$ and hence $Q_n(y_{k,n+1})$ and $Q_n(y_{k+1,n+1})$ have opposite signs. It follows from \eqref{eq:parity} that \[\operatorname{sign}\!\left(F_{n+2}(y_{k,n+1})\right)=-\operatorname{sign}\!\left(F_{n+2}(y_{k+1,n+1})\right).\] Since $P_n$ has real and simple zeros, $E_1<E_2$, and neither $E_i$ is a zero of $P_n$, all zeros of $F_{n+2}$ are real and simple. Hence, by the above equation, $F_{n+2}$ has an odd number of zeros in $I_k$.

\medskip \noindent It remains to consider the two exterior intervals. Because $Q_n$ has no zero in $I_0$, the value of $Q_n(y_{1,n+1})$ has the same sign as $Q_n(x)$ for sufficiently large negative $x$. Moreover, $Q_n$ and $F_{n+2}$ are monic and their degrees differ by two. They therefore have the same sign for sufficiently large negative $x$. By \eqref{eq:endpoint-product-parity}, $F_{n+2}(y_{1,n+1})$ has the same sign as the limiting sign of $F_{n+2}$ at the left-hand end when $\varepsilon =1$, and the opposite sign when $\varepsilon = - 1$. Thus $F_{n+2}$ has an even number of zeros in $I_0$ when $\varepsilon =1$, and an odd number when $\varepsilon = -1$.

\medskip \noindent
Similarly, $Q_n$ has no zero in
\[
I_{n+1}=(y_{n+1,n+1},\infty),
\]
so $Q_n(y_{n+1,n+1})$ has the same sign as $Q_n(x)$ for sufficiently large positive $x$. Since $Q_n$ and $F_{n+2}$ are monic, both are positive for sufficiently large positive $x$. It follows from \eqref{eq:parity} that $F_{n+2}(y_{n+1,n+1})$ has the same sign as $F_{n+2}(x)$ for sufficiently large positive $x$ when $\varepsilon=1$, and the opposite sign when $\varepsilon=-1$. Hence $F_{n+2}$ has an even number of zeros in $I_{n+1}$ when $\varepsilon=1$, and an odd number of zeros in $I_{n+1}$ when $\varepsilon=-1$. This completes the proof.
\end{proof}


\begin{theorem}\label{MainTheorem3} Let $P_n$, $G_{n+1}$, and $Q_n$ be monic polynomials of degrees $n$, $n+1$, and $n$, respectively, with real and simple zeros. Assume that $G_{n+1}\prec Q_n$ and that, for all real $x$, \begin{equation} A(x)P_n(x) = B(x)G_{n+1}(x) - (x-E_1)(x-E_2)Q_n(x), \qquad n\geq1, \label{GeneralMixedTTRR3} \end{equation} where $E_1,E_2\in\mathbb{R}$, $E_1<E_2$, and $A(x)>0$. Suppose also that $B(E_i)\neq0$, $ i=1,2$. If $P_n$ and $G_{n+1}$ are coprime, then \[ (x-E_1)(x-E_2)P_n(x)\prec G_{n+1}(x). \] If, in addition, both $E_1$ and $E_2$ lie outside the interval determined by the extreme zeros of $G_{n+1}$, then necessarily \[ E_1<y_{1,n+1}, \qquad E_2>y_{n+1,n+1}, \] and \[ G_{n+1}(x)\prec P_n(x). \] If $P_n$ and $G_{n+1}$ are not coprime, then every common zero belongs to $\{E_1,E_2\}$. In particular, $P_n$ and $G_{n+1}$ can have at most two common zeros. \end{theorem} 

\begin{proof}
We first prove the assertion concerning common zeros. Suppose that $x_0$ is a
common zero of $P_n$ and $G_{n+1}$. Substituting $x=x_0$ into \eqref{GeneralMixedTTRR3}
gives
\[
(x_0-E_1)(x_0-E_2)Q_n(x_0)=0.
\]
Since $G_{n+1}\prec Q_n$, the polynomials $G_{n+1}$ and $Q_n$ have no
common zero. Hence $Q_n(x_0)\neq 0$, and therefore $x_0\in\{E_1,E_2\}$.
Thus
\[
Z(P_n)\cap Z(G_{n+1})\subseteq\{E_1,E_2\},
\]
where $Z(p)$ denotes the zero set of the polynomial $p$. Therefore, $P_n$ and $G_{n+1}$ can have at most two common zeros. Now assume that $P_n$ and $G_{n+1}$ are coprime, and set $F_{n+2}(x)=(x-E_1)(x-E_2)P_n(x).$
Applying Lemma~\ref{lem:parity-counting} with $\varepsilon=-1$, we find that $F_{n+2}$ has an
odd number of zeros in each of the intervals $I_0,I_1,\ldots,I_n,I_{n+1}.$
In particular, each of these $n+2$ disjoint intervals contains at least one
zero of $F_{n+2}$. Since $F_{n+2}$ has degree $n+2$, it must have exactly
one zero in each interval. Consequently,
\[
(x-E_1)(x-E_2)P_n(x)\prec G_{n+1}(x).
\] Suppose, in addition, that both $E_1$ and $E_2$ lie outside the interval
determined by the extreme zeros of $G_{n+1}$. Since each exterior interval
contains exactly one zero of $F_{n+2}$, the two completion points cannot lie
in the same exterior interval. Because $E_1<E_2$, it follows that
\[
E_1<y_{1,n+1},
\qquad
E_2>y_{n+1,n+1}.
\]
The unique zero of $F_{n+2}$ in $I_0$ is therefore $E_1$, and the unique
zero in $I_{n+1}$ is $E_2$. Hence $P_n$ has no zero in either exterior
interval. Moreover, neither completion point lies in an interior interval, so
the unique zero of $F_{n+2}$ in each $I_k$, $1\leq k\leq n$, is a zero
of $P_n$. Hence $P_n$ has exactly one zero in each interval
\[
(y_{k,n+1},y_{k+1,n+1}), \qquad k=1,\ldots,n,
\]
and no zero in either exterior interval. Therefore,  $G_{n+1}(x)\prec P_n(x).$
\end{proof}
\begin{remark}\label{Conclusion-Proof-MainTheorem3}
The proof of Theorem~\ref{MainTheorem3} also shows that $E_1$ and $E_2$ cannot lie
in the same one of the intervals
\[
I_0,I_1,\ldots,I_n,I_{n+1}.
\]
Although the two-point factor gives the completed interlacing
conclusion, the proof also identifies the following configurations in
which one completion point alone suffices:
		\begin{enumerate}[label=\alph*)]
			\item 
			If $E_1<y_{1,n+1}$ and $y_{k,n+1}<E_2<y_{k+1,n+1}$ for some $k\in\{1,\dots,n\}$, then $G_{n+1}(x)\prec (x-E_2)P_n(x)$.
			\item 
			If $E_2>y_{n+1,n+1}$ and $y_{k,n+1}<E_1<y_{k+1,n+1}$ for some $k\in\{1,\dots,n\}$, then $(x-E_1)P_n(x)\prec G_{n+1}(x)$.
		\end{enumerate}

\end{remark}
 
\begin{theorem}\label{MainTheorem4} Assume that all the hypotheses of Theorem \ref{MainTheorem3} hold, except that \eqref{GeneralMixedTTRR3} is replaced by \begin{equation} A(x)P_n(x) = B(x)G_{n+1}(x) + (x-E_1)(x-E_2)Q_n(x). \label{GeneralMixedTTRR4} \end{equation} Assume that $P_n$ and $G_{n+1}$ are coprime. Then the following statements hold. \begin{enumerate} \item[(a)] The configuration $ E_1<y_{1,n+1}$ and $E_2>y_{n+1,n+1} $ is impossible. \item[(b)] If $E_1$ and $E_2$ lie in the same one of the intervals $I_0,I_1,\ldots,I_n,I_{n+1}$, then $G_{n+1}(x)\prec P_n(x)$. \item[(c)] If $ E_1<y_{1,n+1} $ and $ y_{k,n+1}<E_2<y_{k+1,n+1} $ for some $k\in\{1,\ldots,n\}$, then $(x-E_2)P_n(x)\prec G_{n+1}(x).$ \item[(d)] If $E_2>y_{n+1,n+1}$ and $y_{k,n+1}<E_1<y_{k+1,n+1}$ for some $k\in\{1,\ldots,n\}$, then \newline $G_{n+1}(x)\prec (x-E_1)P_n(x)$. 
\item[(e)] Let $z_{1,n}<z_{2,n}<\cdots<z_{n,n}$ be the zeros of $P_n$.   Suppose that $E_1$ and $E_2$ lie in two distinct interior intervals \[ y_{k,n+1}<E_1<y_{k+1,n+1}, \qquad y_{k',n+1}<E_2<y_{k'+1,n+1}, \qquad k\neq k'. \] Then, for $l\in\{1,\ldots,n-1\}$, \begin{itemize}\item[(i)]    $(x-E_1)G_{n+1}(x)\prec (x-E_2)P_n(x)$  whenever $
				y_{k,n+1} < z_{l,n} < E_1 < z_{l+1,n} < y_{k+1,n+1}
				$, 
				\item[(ii)] $(x-E_2)G_{n+1}(x)\prec (x-E_1)P_n(x)$  whenever $
				y_{k',n+1} < z_{l,n} < E_2< z_{l+1,n} < y_{k'+1,n+1}$. 
			\end{itemize}\end{enumerate} \end{theorem} 

\begin{proof}
Set
\[
F_{n+2}(x)=(x-E_1)(x-E_2)P_n(x).
\]
Applying Lemma~\ref{lem:parity-counting} with $\varepsilon=1$, we find that $F_{n+2}$
has an odd number of zeros in each interior interval
\[
I_1,\ldots,I_n
\]
and an even number of zeros in each exterior interval $I_0$ and
$I_{n+1}$. In particular, every interior interval contains at least one zero of
$F_{n+2}$. These account for at least $n$ zeros. Since
$F_{n+2}$ has degree $n+2$, there are only two further zeros.
Moreover, adding zeros to any one interval without changing the parity
prescribed by Lemma~\ref{lem:parity-counting} requires adding them in pairs. Therefore, the
two further zeros must lie in the same interval. We call this interval
the exceptional interval. Thus, either one exterior interval contains
exactly two zeros, or one interior interval contains exactly three
zeros. 

\medskip \noindent Suppose first that
\[
E_1<y_{1,n+1}
\qquad\text{and}\qquad
E_2>y_{n+1,n+1}.
\]
Then each exterior interval contains one of the completion points.
Since the number of zeros in each exterior interval must be even, each
would have to contain at least two zeros. This would require at least
four zeros in the exterior intervals, in addition to at least one zero
in each of the $n$ interior intervals, contradicting
$\deg F_{n+2}=n+2$. This proves part~(a).

\medskip \noindent Suppose next that $E_1$ and $E_2$ lie in the same interval. If they
lie in an exterior interval, that interval already contains two zeros
of $F_{n+2}$, so it is the exceptional interval. If they lie in an
interior interval, that interval contains at least the two zeros
$E_1$ and $E_2$; since its number of zeros must be odd, it contains
at least three zeros and is again the exceptional interval. In
either case, after removing the two completion points, $P_n$ has
exactly one zero in every interior interval and no zero in either
exterior interval. Hence $G_{n+1}(x)\prec P_n(x),$ which proves part~(b).

\medskip \noindent Now suppose that
\[
E_1\in I_0
\qquad\text{and}\qquad
E_2\in I_k
\]
for some $k\in\{1,\ldots,n\}$. Since $I_0$ contains $E_1$ and
must contain an even number of zeros of $F_{n+2}$, it contains at
least one further zero. Thus $I_0$ is the exceptional interval.
It follows that $I_0$ contains exactly two zeros, every interior
interval contains exactly one zero, and $I_{n+1}$ contains none. The unique zero of $F_{n+2}$ in $I_k$ is $E_2$, so $P_n$ has no
zero in $I_k$. The other zero of $F_{n+2}$ in $I_0$, besides
$E_1$, is a zero of $P_n$, and every interior interval other than
$I_k$ contains exactly one zero of $P_n$. Therefore
$(x-E_2)P_n(x)$ has exactly one zero in $I_0$ and in every interior
interval, and no zero in $I_{n+1}$. Hence
\[
(x-E_2)P_n(x)\prec G_{n+1}(x),
\]
which proves part~(c).

\medskip \noindent The proof of part~(d) is analogous. If
\[
E_1\in I_k
\qquad\text{and}\qquad
E_2\in I_{n+1},
\]
then $I_{n+1}$ is the exceptional interval. Thus
$(x-E_1)P_n(x)$ has no zero in $I_0$, exactly one zero in every
interior interval, and exactly one zero in $I_{n+1}$. Consequently,
\[
G_{n+1}(x)\prec (x-E_1)P_n(x).
\]

\medskip \noindent Finally, suppose that
\[
E_1\in I_k,
\qquad
E_2\in I_{k'},
\qquad
k\neq k'.
\] Each of the intervals $I_k$ and $I_{k'}$ already contains one zero
of $F_{n+2}$, namely $E_1$ and $E_2$, respectively. Every other
interior interval contains at least one zero of $P_n$. Hence
$n-2$ zeros of $P_n$ are forced, one in each of the remaining
$n-2$ interior intervals. The two remaining zeros of $P_n$ must
occur together in the exceptional interval. If these two zeros lie in $I_k$ and satisfy
\[
y_{k,n+1}<z_{l,n}<E_1<z_{l+1,n}<y_{k+1,n+1},
\]
then the zero $E_1$ divides $I_k$ into two intervals, each
containing exactly one zero of $P_n$. Every other interval between
consecutive zeros of $(x-E_1)G_{n+1}$ contains exactly one zero of
$(x-E_2)P_n$; in $I_{k'}$, that zero is $E_2$. Therefore, $(x-E_1)G_{n+1}(x)\prec (x-E_2)P_n(x)$, which proves part~(e)(i).

\medskip \noindent Similarly, if the two remaining zeros of $P_n$ lie in $I_{k'}$ and
satisfy
\[
y_{k',n+1}<z_{l,n}<E_2<z_{l+1,n}<y_{k'+1,n+1},
\]
then $(x-E_2)G_{n+1}(x)\prec (x-E_1)P_n(x)$. This proves part~(e)(ii) and completes the proof.
\end{proof}

 \section{Wronskian-type bounds from completed interlacing} \label{CI and Wronskian-Turan}
 
The purpose of this section is to derive analytic inequalities from the completed
interlacing obtained in Theorem~\ref{MainTheorem3}.  The role of interlacing
in two-parameter Tur\'an-type inequalities has previously been considered in
the ultraspherical and Laguerre polynomial setting~\cite{BustozJMMA1981,BustozIsmail1983SIAM}. In the present setting, the relevant
cross-family pair does not, in general, interlace before completion. The
two-point completed interlacing established in
Theorem~\ref{MainTheorem3} restores the required zero structure through the
augmented polynomial $	F_{n+2}(x)=(x-E_1)(x-E_2)P_n(x).$ 
To translate this discrete geometric property into a continuous analytic
statement, we consider the Wronskian of $G_{n+1}$ and $F_{n+2}$. For differentiable functions $f$ and $g$, we define their Wronskian by $W(f,g)(x):=f(x)g'(x)-f'(x)g(x).$
The following lemma determines the Wronskian's sign, which will be helpful in deriving the cross-family inequalities.
\begin{lemma}\label{lem:Wronskiangeneral}
	Suppose the conditions of Theorem \ref{MainTheorem3} hold and, in addition, the polynomials $P_n$ and $
	G_{n+1}$ are coprime. Let \[F_{n+2}(x)=(x-E_1)(x-E_2)P_n(x).\] Then, for any $x\in \mathbb{R}$ and $n\geq 1$,
	\begin{equation}\label{Wronskian_Neg_twopt}
		W(G_{n+1},F_{n+2} )(x)>0.
	\end{equation}
\end{lemma}

\begin{proof} Since the positions of $E_1$ and $E_2$ are not fixed, we do not retain the notation of zeros of $P_n$ as mentioned in Theorem \ref{MainTheorem4}. For simplicity, we denote the zeros of $F_{n+2}$ by $z_1<z_2,\dots<z_{n+2}$. By Theorem \ref{MainTheorem3}, the $n+2$ distinct real zeros of $F_{n+2}$ strictly interlace the zeros of $G_{n+1}$.
	For $x\in\mathbb{R}\setminus \{z_1,z_2,\dots,z_{n+2}\}$, we define the rational function 
	\begin{equation*}
		h(x):=\frac{G_{n+1}(x)}{F_{n+2}(x)}.
	\end{equation*}
		Since $F_{n+2}$ is monic with simple zeros and
	$\deg G_{n+1}<\deg F_{n+2}$, $h$ has the partial fraction decomposition 
	\begin{equation}\label{lmeq1:pfd}
		h(x)=\sum_{j=1}^{n+2}\frac{R_j}{x-z_j},
		\qquad
		R_j=\frac{G_{n+1}(z_j)}{F_{n+2}'(z_j)},
	\end{equation}
    The interlacing pattern shows that $G_{n+1}(z_j)$ and $F'_{n+2}(z_j)$ have the same sign. Indeed \[\operatorname{sign}[G_{n+1}(z_j)]=\operatorname{sign}[F_{n+2}'(z_j)]=(-1)^{n+2-j}.\]  Hence  
	\begin{equation*}
		R_j>0,\qquad j=1,\ldots,n+2.
	\end{equation*}
	Differentiating the partial-fraction expansion \eqref{lmeq1:pfd} with respect to $x$ gives
	\begin{equation}\label{lemeq2:hprime-pfd}
		h'(x)=-\sum_{j=1}^{n+2}\frac{R_j}{(x-z_j)^2}<0.
	\end{equation} On the other hand,
	\begin{equation}\label{lemeq0:hprime}
		h'(x)=\frac{G'_{n+1}(x)F_{n+2}(x)-G_{n+1}(x)F'_{n+2}(x)}{F_{n+2}^2(x)}
		=-\frac{	W(G_{n+1},F_{n+2})(x)}{F^2_{n+2}(x)}.
	\end{equation}	By \eqref{lemeq2:hprime-pfd} and \eqref{lemeq0:hprime}, \[W(G_{n+1},F_{n+2})(x)=F^2_{n+2}(x)\,\sum_{j=1}^{n+2}\frac{R_j}{(x-z_j)^2} >0\]
	whenever $x\ne z_j$. 	At each root $x = z_j$, the Wronskian simplifies to \[W(G_{n+1},F_{n+2})(z_j) = F_{n+2}'(z_j)G_{n+1}(z_j)>0\] because the two factors have the same nonzero sign. 
	Hence $W(G_{n+1},F_{n+2})(x)>0$ on $\mathbb{R}$.
\end{proof}
\begin{remark} If the coprimality assumption in Lemma \ref{lem:Wronskiangeneral} is relaxed, then, by Theorem~\ref{MainTheorem3}, every common zero of $P_n$ and $G_{n+1}$ belongs to $\{E_1,E_2\}$. If $x_0$ is such a common zero, then  $	W(G_{n+1},F_{n+2})(x_0)=0$.  \end{remark}

Motivated by classical Tur\'{a}n inequalities, define the
generalized cross-family differential Tur\'{a}nian
\[
\mathcal{T}_n(x):=
(2x-E_1-E_2)P_n(x)G_{n+1}(x)
-(x-E_1)(x-E_2)W(P_n,G_{n+1})(x)
-G_{n+1}^2(x).
\]
Next, we derive Wronskian-type bounds from completed interlacing. The argument requires neither orthogonality nor any special structure of the coefficient functions $A$ and $B$ in the mixed recurrence relation \eqref{GeneralMixedTTRR3}.

\begin{theorem}\label{thm:general}
	Assume that the hypotheses of Theorem \ref{MainTheorem3} hold and that $P_n$ and $
	G_{n+1}$ are coprime. Then \[\mathcal{T}_n(x)>0, \qquad x\in\mathbb{R}.\]
Equivalently, the weighted Wronskian satisfies the Wronskian-type bounds
	\begin{equation}
	\label{eq:general-turan-sharp}
	(x-E_1)(x-E_2)W(P_n,G_{n+1})(x)
	<
	(2x-E_1-E_2)P_n(x)G_{n+1}(x)-G_{n+1}(x)^2,
\end{equation}
	 and 
		\begin{equation}\label{eq:general-turan}
		(x-E_1)(x-E_2)W(P_n,G_{n+1})(x)
		\;<\;(2x-E_1-E_2)P_n(x)G_{n+1}(x) 
	\end{equation}
 for every $x\in\mathbb{R}$. Here, \eqref{eq:general-turan-sharp} is the refined bound and
\eqref{eq:general-turan} is its weaker form.
\end{theorem}

\begin{proof}
Multiplying \eqref{lmeq1:pfd} by $x$ and letting $x\rightarrow \infty$, we obtain $\sum_{j=1}^{n+2}R_j=1$ since $G_{n+1}$ and $F_{n+2}$ are monic and their degrees differ by one. Indeed, \[1=\lim_{x\to\infty}x\frac{G_{n+1}(x)}{F_{n+2}(x)}= \lim_{x\to \infty}\sum_{j=1}^{n+2}R_j\frac{x}{x-z_j}=\sum_{j=1}^{n+2}R_j.\] If $x\in\mathbb{R}\setminus\{z_1,\cdots,z_{n+2}\}$, then  by using \eqref{lmeq1:pfd} and \eqref{lemeq2:hprime-pfd},  we can write
	\begin{align*}
		-h'(x)-h(x)^2
		=
		\sum_{j=1}^{n+2}\frac{R_j}{(x-z_j)^2}
		-
		\left(
		\sum_{j=1}^{n+2}\frac{R_j}{x-z_j}
		\right)^2                                             
		=
		\sum_{1\leq j<k\leq n+2}
		R_jR_k
		\left(
		\frac{1}{x-z_j}-\frac{1}{x-z_k}
		\right)^2.
	\end{align*}
	By~\eqref{lemeq0:hprime} and the definition of $h$, this identity becomes
	\begin{align*}
		W(G_{n+1},F_{n+2})(x)	-
		G_{n+1}(x)^2                                    
		=
		F_{n+2}(x)^2
		\sum_{1\leq j<k\leq n+2}
		R_jR_k
		\frac{(z_j-z_k)^2}
		{(x-z_j)^2(x-z_k)^2}.
	\end{align*}
	Since $F_{n+2}(x)=(x-E_1)(x-E_2)P_n(x)$ is monic with zeros $z_1,\ldots,z_{n+2}$, we have
	\begin{equation}\label{eq:exact-positive-remainder-general}
		\begin{aligned}
			W(G_{n+1},F_{n+2})(x)	-G_{n+1}(x)^2                                    
			=
			\sum_{1\leq j<k\leq n+2}
			R_jR_k(z_j-z_k)^2
			\left(
			\prod_{\substack{\ell=1\\ \ell\neq j,k}}^{n+2}
			(x-z_{\ell})
			\right)^2.
		\end{aligned}
	\end{equation}
	Both sides of ~\eqref{eq:exact-positive-remainder-general} are polynomials, so the identity extends to every
\(x\in\mathbb{R}\). Moreover, the right-hand side of~\eqref{eq:exact-positive-remainder-general}
	is strictly positive for every real $x$. Indeed, if $x\notin\{z_1,\ldots,z_{n+2}\}$, then every summand is
strictly positive. If $x=z_m$, then every summand whose index pair contains $m$, for example the pair $(m,k)$ with
$k\neq m$, is strictly positive. Hence
	\begin{equation}\label{eq:Wronskian-square-improvement}
		W(G_{n+1},F_{n+2})(x)
		>
		G_{n+1}(x)^2,
		\qquad x\in\mathbb{R}.
	\end{equation} Now,
\[W(G_{n+1},F_{n+2})(x)=(2x-E_1-E_2)P_n(x)G_{n+1}(x) -(x-E_1)(x-E_2)W(P_n,G_{n+1})(x).\]
Substitution into \eqref{eq:Wronskian-square-improvement} and rearrangement yield \eqref{eq:general-turan-sharp}.
Finally, \eqref{eq:general-turan} follows from \eqref{eq:general-turan-sharp} because \(G_{n+1}^2(x)\ge0\). 
\end{proof}

\begin{remark} \begin{itemize} \item[]
\item[(i)]The weaker bound \eqref{eq:general-turan} follows directly from Lemma~\ref{lem:Wronskiangeneral}, without the refinement used to prove \eqref{eq:general-turan-sharp}. Indeed, differentiating $F_{n+2}(x)=(x-E_1)(x-E_2)P_n(x)$, we get
	$F_{n+2}'(x)=(2x-E_1-E_2)P_n(x)+(x-E_1)(x-E_2)P'_n(x)$. Now,  using \eqref{lem:Wronskiangeneral}
	\begin{align}
		\nonumber			0&<W(G_{n+1},F_{n+2})(x)=(2x-E_1-E_2)P_n(x)G_{n+1}(x) -(x-E_1)(x-E_2)W(P_n,G_{n+1})(x)
	\end{align}
	Rearranging the above equation yields \eqref{eq:general-turan}.	
\item[(ii)] In Section~\ref{Turan inequalities}, inequalities \eqref{eq:general-turan-sharp} and \eqref{eq:general-turan}
will be specialized to the doubly shifted Jacobi pair
\(P_n^{(\alpha,\beta)}\) and
\(P_{n+1}^{(\alpha+1,\beta+1)}\), yielding a sharpened cross-family Turán-type inequality and its simpler form.  
\end{itemize}
\end{remark}

The Wronskian-type bounds obtained in Theorem \ref{thm:general} hold  for any two polynomials whose zeros interlace after adjoining two completion points. We now consider the additional structure that arises when the polynomial $Q_n$ in the mixed recurrence relation belongs to the same orthogonal sequence as $G_{n+1}$. 	In this setting, differentiating the mixed recurrence gives an identity that relates the cross-family Wronskian $W(P_n,G_{n+1})$ to the same-family Wronskian $W(G_n,G_{n+1})$. Orthogonality then determines the sign of the latter through the Christoffel–Darboux formula and consequently yields two-sided bounds for the cross-family Wronskian.
\begin{proposition}\label{lem:generaldecomposition}	Suppose $\{G_k\}$ is orthogonal with respect to a positive measure supported on 
	$(a,b)$. Assume that the hypotheses of Theorem \ref{MainTheorem3}  
    hold with $Q_n=G_n$, $A$ and $B$ differentiable on $\mathbb{R}$ and $P_n$ and $G_{n+1}$ coprime. Define \[R(x):=A'(x)P_n(x)G_{n+1}(x)-B'(x)G_{n+1}(x)^2+(2x-E_1-E_2)G_n(x)G_{n+1}(x).\]
    Then the following statements hold:
    \begin{itemize}
\item[(a)]For every $x\in\mathbb{R}$,
	\begin{equation}\label{eq1:generaldecomposition}
		A(x) W(P_n,G_{n+1})(x)=R(x)-(x-E_1)(x-E_2)W(G_n,G_{n+1})(x)
	\end{equation}
\item[(b)] For  $x\in\mathbb{R}$ with $x\ne E_1,E_2$
	\begin{equation}\label{eq:two-sided}
		A(x)W(P_n,G_{n+1})(x)\;
		\begin{cases}
			> R(x) & x\in(E_1,E_2),\\
			<R(x) & x\in(-\infty,E_1)\cup (E_2,\infty).
		\end{cases}
	\end{equation}
	 Equality holds if and only if $x\in\{E_1,E_2\}$.
\item[(c)] For all $x\in \mathbb{R}$
	\begin{equation}\label{eq:pos-bound}
		A(x)(2x-E_1-E_2)P_n(x)G_{n+1}(x)
		>(x-E_1)(x-E_2)R(x)-(x-E_1)^2(x-E_2)^2W(G_n,G_{n+1})(x).\end{equation}
	\end{itemize}
\end{proposition}

\begin{proof} Since the mixed recurrence \eqref{GeneralMixedTTRR3} holds with $Q_n=G_n$, we have 
	\begin{align}\label{GeneralMixedTTRR3:QnequalsGn}
		A(x)\,P_{n}(x)=B(x)\,G_{n+1}(x)-(x-E_1)(x-E_2)\,G_{n}(x), 
	\end{align}
	Differentiating \eqref{GeneralMixedTTRR3:QnequalsGn} gives
	\begin{equation}\label{eq:diff-id}
		A'(x)P_n(x)+A(x)P_n'(x)=B'(x)\,G_{n+1}(x)+B(x)G_{n+1}'(x)-(2x-E_1-E_2)G_n(x)-(x-E_1)(x-E_2)G_n'(x).
	\end{equation}
	Now we obtain \eqref{eq1:generaldecomposition} by substituting  \eqref{GeneralMixedTTRR3:QnequalsGn} into the first term and
	\eqref{eq:diff-id} into the second term on the right of the expression \[A(x)\,W(P_n,G_{n+1})(x)=A(x)\,P_n(x)\cdot G_{n+1}'(x)-A(x)\,P_n'(x)\cdot G_{n+1}(x).\] This proves part (a).

\medskip \noindent Since $\{G_k\}$ is orthogonal with respect to a positive measure, the Christoffel-Darboux formula guarantees the strict positivity of the classical Wronskian $W(G_n,G_{n+1})(x) > 0$ for $x\in \mathbb{R}$. It follows from part (a) that 
\[
		A(x) W(P_n,G_{n+1})(x)-R(x)=-(x-E_1)(x-E_2)W(G_n,G_{n+1})(x).
	\] The right-hand side is positive when $x\in(E_1,E_2)$, negative when $x\in(-\infty,E_1)\cup (E_2,\infty)$ and zero precisely when $x=E_1$ and $x=E_2$. This proves part (b).

\medskip \noindent To prove~\eqref{eq:pos-bound}, multiply~\eqref{eq:general-turan} by $A(x)>0$ to obtain
	\begin{align*}
		A(x)(x-E_1)(x-E_2)W(P_n,G_{n+1})(x) < A(x)(2x-E_1-E_2)P_n(x)G_{n+1}(x).
	\end{align*}
	Substitute the expression of $A(x)W(P_n,G_{n+1})(x)$ derived in \eqref{eq1:generaldecomposition}  into the left-hand side to obtain
	\begin{align*}
		(x-E_1)(x-E_2)\bigl[R(x) - (x-E_1)(x-E_2)W(G_n,G_{n+1})(x)\bigr] < A(x)(2x-E_1-E_2)P_n(x)G_{n+1}(x).
	\end{align*}
	Rearranging completes the proof of (c). 

\end{proof}

	 \section{Jacobi polynomials} \label{sec:Jacobianalysis and Inequalities}
	
	\subsection{Two-point completion of zero interlacing between different Jacobi families}\label{Jacobi polynomials}
	
	The monic Jacobi polynomials, denoted by $P_n^{(\alpha, \beta)}(x)$, form an orthogonal sequence on $(-1, 1)$ with respect to the measure $d\mu(x) = (1 - x)^\alpha (1 + x)^\beta \, dx$, where  $\alpha, \beta > -1$ \cite{Chiharabook}. They satisfy the three-term recurrence relation
	\begin{align}\label{TTRR-Monic-Jacobi}
		P_{n+1}^{(\alpha, \beta)}(x) = (x - c^{(\alpha, \beta)}_{n+1}) P_n^{(\alpha, \beta)}(x) - \lambda^{(\alpha, \beta)}_{n+1} P_{n-1}^{(\alpha, \beta)}(x),
	\end{align}
	
	with $P_{-1}^{(\alpha, \beta)}(x) = 0$ and $P_0^{(\alpha, \beta)}(x) = 1$, where the recurrence  coefficients $c^{(\alpha, \beta)}_{n+1}$ and $\lambda^{(\alpha, \beta)}_{n+1}$ are given by
	\begin{align}
		\nonumber \lambda^{(\alpha, \beta)}_{n+1} &= \frac{4n(n + \alpha)(n + \beta)(n + \alpha + \beta)}{(2n + \alpha + \beta)^2 (2n + \alpha + \beta + 1)(2n + \alpha + \beta - 1)}>0, \qquad n\geq1,\\ c^{(\alpha, \beta)}_{n+1}& = \frac{\beta^2 - \alpha^2}{(2n + \alpha + \beta)(2n + \alpha + \beta + 2)}\in \mathbb{R},
	\end{align}
    	except that
	\begin{align*}
	c_1^{(\alpha,\beta)}
			=
			\frac{\beta-\alpha}{\alpha+\beta+2}
			\quad\text{when }\alpha=-\beta,
			\qquad
			\lambda_2^{(\alpha,\beta)}
			=
			2(\alpha+1)(\beta+1)
			\quad\text{when }\alpha+\beta=-1.
	\end{align*}

    
		\begin{lemma}
		\label{LemmaLocationE1E2}
		Let $n\geq1$ and $\alpha,\beta>-1$. Define
		\begin{align}
			E_1:=E_{n,\alpha,\beta}
			&=
			\frac{1}{2}
			\left[
			\frac{2(n+1)(\alpha-\beta)}
			{(2n+\alpha+\beta+4)(2n+\alpha+\beta+2)}
			-
			\sqrt{\Delta}
			\right],
			\label{eq:E1}\\
			E_2:=F_{n,\alpha,\beta}
			&=
			\frac{1}{2}
			\left[
			\frac{2(n+1)(\alpha-\beta)}
			{(2n+\alpha+\beta+4)(2n+\alpha+\beta+2)}
			+
			\sqrt{\Delta}
			\right],
			\label{eq:E2}
		\end{align}
		where
		\begin{multline}
			\label{eq:Delta}
			\Delta =	\left(
			\frac{2(n+1)(\alpha-\beta)}	{(2n+\alpha+\beta+4)(2n+\alpha+\beta+2)}	\right)^2	\\	+	\frac{4}{(2n+\alpha+\beta+2)^2}	\left[	\frac{(\alpha-\beta)^2(\alpha+\beta+2)}
			{2n+\alpha+\beta+4}	+	\frac{4(n+\alpha+1)(n+\beta+1)(n+\alpha+\beta+2)}
			{2n+\alpha+\beta+3}
			\right].
		\end{multline}
		Then $\Delta>0$, and hence $E_1$ and $E_2$ are real and distinct. Moreover, $	E_1<E_2$ and $-1<E_1<c_{n,\alpha,\beta}<E_2<1$, where 
		$\displaystyle{c_{n,\alpha,\beta}
			=
			\frac{\alpha-\beta}{2n+\alpha+\beta+2}}.$
		In particular, $-1<E_1<0<E_2<1.$
	\end{lemma}
	
	\begin{proof}
		For $n\geq1$ and $\alpha,\beta>-1$, the second term inside the square brackets of \eqref{eq:Delta} is strictly positive. Thus
		the square bracket is strictly positive, and consequently $	\Delta>0.$
		Therefore $E_1$ and $E_2$ are real. Also, by
		\eqref{eq:E1} and \eqref{eq:E2}
		\begin{align*}
			E_2-E_1=\sqrt{\Delta}>0 ~\implies ~	E_1<E_2.
		\end{align*}
	Now, we put $\displaystyle  D=2n+\alpha+\beta+2,
\text{and}~
	c_{n,\alpha,\beta}	=
	\frac{\alpha-\beta}{D}.$
	We can write
		\begin{align*}
			D^2-(\alpha-\beta)^2=	\left(D-(\alpha-\beta)\bigr)\bigl(D+(\alpha-\beta)\right)
	=	4(n+\beta+1)(n+\alpha+1)>0.
		\end{align*}
	Hence $|c_{n,\alpha,\beta}|<1$.	Let
		\begin{align*}
			T_{n,\alpha,\beta}(x)=(x-E_1)(x-E_2).
		\end{align*}
Next, evaluating the quadratic $T_{n,\alpha,\beta}$ at $-1$, $c_{n,\alpha,\beta}$, and $1$,
we obtain
		\begin{align*}
			T_{n,\alpha,\beta}(-1)	&=	\frac{	4(n+1)(n+\alpha+1)(n+\alpha+2)	}{
				(2n+\alpha+\beta+2)(2n+\alpha+\beta+3)(2n+\alpha+\beta+4)	}>0,\\[4pt]
			T_{n,\alpha,\beta}(1)&=	\frac{4(n+1)(n+\beta+1)(n+\beta+2)}{	(2n+\alpha+\beta+2)(2n+\alpha+\beta+3)(2n+\alpha+\beta+4)}>0,
		\end{align*}
		while
		\begin{align*}
			T_{n,\alpha,\beta}(c_{n,\alpha,\beta})&=	-\frac{	4(n+\alpha+1)(n+\beta+1)(n+\alpha+\beta+2)}{
				(2n+\alpha+\beta+2)^2(2n+\alpha+\beta+3)}<0.
		\end{align*}
The intermediate value theorem implies that one zero of
		$T_{n,\alpha,\beta}$ lies in $(-1,c_{n,\alpha,\beta})$ and the other lies in $(c_{n,\alpha,\beta},1)$. Since
		$E_1<E_2$, we obtain $-1<E_1<c_{n,\alpha,\beta}<E_2<1.$

To show that the two roots do not have the same sign, we see that
		\begin{align*}
			E_1E_2	&=	-\frac{1}{(2n+\alpha+\beta+2)^2}	\left[
			\frac{(\alpha-\beta)^2(\alpha+\beta+2)}	{2n+\alpha+\beta+4}	+
			\frac{4(n+\alpha+1)(n+\beta+1)(n+\alpha+\beta+2)}	{2n+\alpha+\beta+3}	\right].
		\end{align*}
		The expression inside the square brackets is strictly positive for
		$n\geq1$ and $\alpha,\beta>-1$. Therefore $	E_1E_2<0.$
		Since $E_1<E_2$, this implies $	E_1<0<E_2.$
		Combining this with $E_1,E_2\in(-1,1)$ gives
		\begin{align*}
			-1<E_1<0<E_2<1.
		\end{align*}
		This completes the proof.
	\end{proof}

	\begin{corollary}\label{Interlace-Jacobi-both-shifted-degree-parameter}
		Let $P^{(\alpha, \beta)}_n(x)$ be a monic Jacobi  polynomial of degree $n$ with $\alpha>-1,\beta>-1$. Let  $\{y_{k,n+1}\}_{k=1}^{n+1}$ denote the zeros of the polynomial  $P^{(\alpha+1, \beta+1)}_{n+1}(x)$. Assume that $P^{(\alpha, \beta)}_n(x)$ and $P^{(\alpha+1, \beta+1)}_{n+1}(x)$ have no common zero. Then,  $(x - E_1)(x - E_2) P^{(\alpha, \beta)}_n(x)\prec P^{(\alpha+1, \beta+1)}_{n+1}(x)$, where $E_1$ and $E_2$ are given in \eqref{eq:E1} and \eqref{eq:E2}. In particular, if both $E_1$ and $E_2$ lie outside the extreme zeros of $P^{(\alpha+1, \beta+1)}_{n+1}(x)$, then $E_1 < y_{1,n+1}$, $E_2 > y_{n+1,n+1}$ and  $P^{(\alpha+1, \beta+1)}_{n+1}(x) \prec P^{(\alpha, \beta)}_{n}(x)$.
	\end{corollary}

	\begin{proof}
		Using (cf.~\cite[eq.~(8)]{Driver-Jordaan-NumerAlg-2018}), the monic Jacobi polynomial $P^{(\alpha,\beta)}_n(x)$ satisfies the following mixed three-term recurrence relation\\
		
		$\displaystyle \frac{n+\alpha+\beta+1}{n}P^{(\alpha, \beta)}_{n}(x)$
		\begin{align}\label{Mixed-TTRR-Monic-Jacobi-DJ-2018-eq8}
			=- \left(x - \frac{\alpha-\beta}{2n+\alpha+\beta+2}\right) P^{(\alpha+1, \beta+1)}_{n-1}(x)+\frac{2n+\alpha+\beta+1}{n} P^{(\alpha+1, \beta+1)}_{n}(x) 
		\end{align}
		for $n \in \mathbb{N}$ and $\alpha, \beta > -1$. Evaluating the recurrence \eqref{TTRR-Monic-Jacobi} at the shifted parameters $(\alpha+1,\beta+1)$ and substituting the expression for $P_{n-1}^{(\alpha+1,\beta+1)}(x)$ from \eqref{Mixed-TTRR-Monic-Jacobi-DJ-2018-eq8}, followed by algebraic simplification, we obtain\\
		
		$\displaystyle 	\frac{(n+\alpha+\beta+1)\,\lambda_{n+1}^{(\alpha+1,\beta+1)}}{n}\;
		{P}_{n}^{(\alpha,\beta)}(x)
		$
		\begin{align}
			\label{Mixed-TTRR-Monic-Jacobi1}
			=\left(x - \frac{\alpha-\beta}{2n+\alpha+\beta+2}\right)
			{P}_{n+1}^{(\alpha+1,\beta+1)}(x)
			- T_{n,\alpha,\beta}(x)\;
			{P}_{n}^{(\alpha+1,\beta+1)}(x),
		\end{align}
		where 
		\begin{align}
			\label{T-coeff-form}
			T_{n,\alpha,\beta}(x)
			=
			\left(x - \frac{\alpha-\beta}{2n+\alpha+\beta+2}\right)
			\left(x - c_{n+1}^{\alpha+1,\beta+1}\right)
			- \frac{2n+\alpha+\beta+1}{n}\lambda_{n+1}^{(\alpha+1,\beta+1)}.
		\end{align}
		Equivalently, $T_{n,\alpha,\beta}(x)$ admits the factored representation
		\begin{align}
			\label{eq:factored_form}
			T_{n,\alpha,\beta}(x) = (x - E_{n,\alpha,\beta})(x - F_{n,\alpha,\beta}),
		\end{align}
		with roots $E_{n,\alpha,\beta}$ and $F_{n,\alpha,\beta}$ given in \eqref{eq:E1} and \eqref{eq:E2}	respectively.	
		Now, \eqref{Mixed-TTRR-Monic-Jacobi1} reads as
		
		$\displaystyle 	\frac{(n+\alpha+\beta+1)\,\lambda_{n+1}^{(\alpha+1,\beta+1)}}{n}\;
		{P}_{n}^{(\alpha,\beta)}(x)
		$
		\begin{align}
			\label{Mixed-TTRR-Monic-Jacobi2}
			=\left(x - \frac{\alpha-\beta}{2n+\alpha+\beta+2}\right)
			{P}_{n+1}^{(\alpha+1,\beta+1)}(x)
			- (x-E_1)(x-E_2)\;
			{P}_{n}^{(\alpha+1,\beta+1)}(x).
		\end{align}
		Let us define  
		$G_{n+1}(x) := {P}_{n+1}^{(\alpha+1,\beta+1)}(x)$,  
		$Q_{n}(x) := {P}_{n}^{(\alpha+1,\beta+1)}(x)$,  
		and $P_{n}(x) := {P}_{n}^{(\alpha,\beta)}(x)$.  
		We also set  
		$A(x) := \frac{(n+\alpha+\beta+1)\,\lambda_{n+1}^{(\alpha+1,\beta+1)}}{n} > 0$  
		and  
		$B(x) := x - \frac{\alpha - \beta}{2n + \alpha + \beta + 2}$. If $B(x)=0$, then $x=\frac{\alpha-\beta}{2n+\alpha+\beta+2}$, and it follows from \eqref{T-coeff-form} that at this point $T(x)=-\frac{2n+\alpha+\beta+1}{n}\lambda_{n+1}^{(\alpha+1,\beta+1)}\neq 0$. Hence no zero of $T$ can be a zero of $B$ and therefore $B(E_i)\neq 0$.
		Furthermore, for any $\alpha, \beta > -1$, the zeros of ${P}_{n}^{(\alpha+1,\beta+1)}(x)$ and ${P}_{n+1}^{(\alpha+1,\beta+1)}(x)$ are strictly interlacing.  Thus, the result follows directly from Theorem~\ref{MainTheorem3}.
	\end{proof}

    	\begin{remark} Assuming $P_n^{(\alpha,\beta)}$ and $P_{n+1}^{(\alpha+1,\beta+1)}$ have no common zeros, it follows from Remark~\ref{Conclusion-Proof-MainTheorem3} that the points $E_1$ and $E_2$, defined in \eqref{eq:E1} and \eqref{eq:E2}, cannot lie simultaneously in the same interval among $(-\infty, y_{1,n+1})$, $(y_{k,n+1}, y_{k+1,n+1})$ for $k=1,\dots,n$, or $(y_{n+1,n+1}, \infty)$, for any $n \in \mathbb{N}$ and $\alpha > -1$, $\beta > -1$, where $\{y_{k,n+1}\}_{k=1}^{n+1}$ are the zeros of $P^{(\alpha+1,\beta+1)}_{n+1}(x)$. Moreover, the following interlacing properties hold using only one of the points
		\begin{enumerate}[label=\alph*)]
			\item 
			If $E_1<y_{1,n+1}$ and $y_{k,n+1}<E_2<y_{k+1,n+1}$ for some $k\in\{1,\dots,n\}$, then $P^{(\alpha+1, \beta+1)}_{n+1}(x)\prec (x-E_2)P^{(\alpha, \beta)}_{n}(x)$.
			\item 
			If $E_2>y_{n+1,n+1}$ and $y_{k,n+1}<E_1<y_{k+1,n+1}$ for some $k\in\{1,\dots,n\}$, then $(x-E_1)P^{(\alpha, \beta)}_{n}(x)\prec P^{(\alpha+1, \beta+1)}_{n+1}(x)$.
		\end{enumerate}
	\end{remark}
	
	\begin{remark}  
		In \cite[Theorem~2 and Remark~6]{Arvesu-Driver-Littlejohn-RamanujanJ-2023}, the authors investigated the relative location of the zeros of 
		${P}^{(\alpha+1,\beta+1)}_{n+1}(x)$ and ${P}^{(\alpha,\beta)}_{n}(x)$ 
		by using the mixed three-term recurrence relation satisfied by the non-monic Jacobi polynomials 
		\cite[eq.~(13)]{Arvesu-Driver-Littlejohn-RamanujanJ-2023} and evaluating it at two consecutive zeros of 
		${P}^{(\alpha,\beta)}_{n}(x)$ (see \cite[eq.~(15)]{Arvesu-Driver-Littlejohn-RamanujanJ-2023}). 
		As a result, the positions of four zeros of 
		${P}^{(\alpha+1,\beta+1)}_{n+1}(x)$ could not be determined analytically in general, 
		and some of their possible configurations were discussed based on numerical observations 
		(see \cite[Remark~6]{Arvesu-Driver-Littlejohn-RamanujanJ-2023}). 
		
		In contrast, in Corollary~\ref{Interlace-Jacobi-both-shifted-degree-parameter}, 
		we employ a mixed three-term recurrence relation for the monic Jacobi polynomials, namely 
		\eqref{Mixed-TTRR-Monic-Jacobi2}, and evaluate it at two consecutive zeros of 
		${P}^{(\alpha+1,\beta+1)}_{n+1}(x)$ rather than those of ${P}^{(\alpha,\beta)}_{n}(x)$. 
		This alternative formulation  simplifies the analysis and leads to a sharper conclusion. 
		In particular, this approach determines the locations of $n-2$ zeros of ${P}^{(\alpha,\beta)}_{n}(x)$, 
		leaving only two zeros undetermined in general. We have analyzed all possible configurations for these two zeros 
		and shown that the positions of the points $E_1$ and $E_2$ determine their locations. 
		Consequently, this yields the precise interlacing of the zeros of 
		${P}^{(\alpha+1,\beta+1)}_{n+1}(x)$ and ${P}^{(\alpha,\beta)}_{n}(x)$ 
		in certain cases, depending on the positions of $E_1$ and $E_2$ for suitable values of $n\in\mathbb{N}$, $\alpha>-1$, and $\beta>-1$. 
		Overall, the present result extends and improves 
		\cite[Theorem~2 and Remark~6]{Arvesu-Driver-Littlejohn-RamanujanJ-2023} 
		by providing a clearer analytical framework for describing the interlacing pattern.
	\end{remark}

	\begin{table}[H]
		\centering
		\caption{Interlacing of the zeros of $P_{n}^{(\alpha,\beta)}$ (denoted by $z_{k,n}$) and $P_{n+1}^{(\alpha+1,\beta+1)}$ (denoted by $y_{k,n+1}$) provided that  $E_1 < y_{1,n+1}$ and $E_2 > y_{n+1,n+1}$, where  $E_1$ and $E_2$ are defined in \eqref{eq:E1} and \eqref{eq:E2} }
		\label{tab1:Interlacingzeros_Jacobi}
		\footnotesize 
		\setlength{\tabcolsep}{12pt} 
		\begin{tabular}{@{}c cc | cc@{}}
			\toprule
			
			\addlinespace[2pt]
			& \multicolumn{2}{c|}{$n=6, \alpha=6, \beta=5$} & \multicolumn{2}{c}{$n=7, \alpha=5, \beta=4$} \\
			& \multicolumn{2}{c|}{$E_1 = -0.84431, E_2=0.86505$} & \multicolumn{2}{c}{$E_1=-0.820009, E_2=0.843713$} \\
			\midrule
			$k$ &  $y_{k,7}$ & $z_{k,6}$ & $y_{k,8}$ & $z_{k,7}$ \\
			\midrule
			1 &  $-0.737759$ & $-0.72289$ & $-0.798587$ &$-0.795866$ \\
			2 &  $-0.527401$ & $-0.475502$  & $-0.614213$ &$-0.582221$\\
			3 & $ -0.293041$ & $-0.197106$ & $-0.400458$ & $-0.329391$ \\
			4 & $-0.0438553$ & $ 0.0958548$  & $-0.166445$ & $-0.0532968$\\
			5 &  $ 0.208718$  &$0.384919$  & $ 0.0766257$  & $ 0.227774$ \\
			6 &  $ 0.453234$  &$0.653855$  & $0.316978$  & $0.495343$  \\
			7 &  $0.680845$  &$-$  & $	0.543279$  & $0.733311$ \\
			8&$-$&$-$&$0.746523$ &$-$\\
			\bottomrule
		\end{tabular}
		
		\vspace{0.2cm}
		\begin{minipage}{0.9\textwidth}
			\footnotesize
			\textbf{Note:} Observe that $E_1 < y_{1,7}$ and $E_2 > y_{7,7}$ for $n=6, \alpha=6, \beta=5$; similarly, $E_1 < y_{1,8}$ and $E_2 > y_{8,8}$ for $n=7, \alpha=5, \beta=4$.
		\end{minipage}
	\end{table}

    \begin{remark}
		Corollary~\ref{Interlace-Jacobi-both-shifted-degree-parameter} also improves the observation of \cite[page 642]{Arvesu-Driver-Littlejohn-RamanujanJ-2023} that ``for $\alpha$ and $\beta$ large compared to $n$, the zeros of ${P}_n^{(\alpha,\beta)}$ and ${P}_{n+1}^{(\alpha+1,\beta+1)}$, $\alpha>-1,\beta>-1$ are interlacing", by establishing that interlacing occurs whenever $E_1 < y_{1,n+1}$ and $E_2 > y_{n+1,n+1}$, where $y_{1,n+1}$ and $y_{n+1,n+1}$ are the smallest and largest zeros of ${P}_{n+1}^{(\alpha+1,\beta+1)}(x)$. In particular, the interlacing of the zeros of ${P}_n^{(\alpha, \beta)}(x)$ and ${P}_{n+1}^{(\alpha+1, \beta+1)}(x)$, with $\alpha > -1$, $\beta > -1$, also occurs when $\alpha$ and $\beta$ are closer to $n$, provided that $E_1 < y_{1,n+1}$ and $E_2 > y_{n+1,n+1}$. This  can also be illustrated by the numerical experiments in Table \ref{tab1:Interlacingzeros_Jacobi}.
	\end{remark}

 \subsubsection{Analysis of completion points and common zeros for the shifted Jacobi pair}

	\begin{proposition}\label{ExactCommonZeroJacobiShifted}
		Let $n\geq 1$ and $\alpha,\beta>-1$. Let
		$E_{1}$ and $E_{2}$ be the two completion
		points defined in \eqref{eq:E1} and \eqref{eq:E2}. Then
		\begin{align*}
			Z(P_n^{(\alpha,\beta)})
			\cap
			Z(P_{n+1}^{(\alpha+1,\beta+1)})
			=
			\left\{
			E_{i}:
			P_{n+1}^{(\alpha+1,\beta+1)}(E_{i})=0,\ i=1,2
			\right\}.
		\end{align*}
		In particular, $P_{n}^{(\alpha,\beta)}$ and $P_{n+1}^{(\alpha+1,\beta+1)}$ have no common zero if and only if $	P_{n+1}^{(\alpha+1,\beta+1)}(E_{1})
		P_{n+1}^{(\alpha+1,\beta+1)}(E_{2})
		\neq0.$
	\end{proposition}
		\begin{proof}
		By Theorem~\ref{MainTheorem3} applied to the mixed recurrence relation
		\eqref{Mixed-TTRR-Monic-Jacobi2}, we have
		\begin{align*}
			Z(P_n^{(\alpha,\beta)})
			\cap
			Z(P_{n+1}^{(\alpha+1,\beta+1)})
			\subseteq
			\left\{
			E_{i}:
			P_{n+1}^{(\alpha+1,\beta+1)}(E_{i})=0,\ i=1,2
			\right\}.
		\end{align*}
			Conversely, suppose that  $	P_{n+1}^{(\alpha+1,\beta+1)}(E_{i})=0$ for some $i\in\{1,2\}$. 
	Substituting
		$x=E_{i}$ into \eqref{Mixed-TTRR-Monic-Jacobi2} gives
		\begin{align*}
			\frac{(n+\alpha+\beta+1)\lambda_{n+1}^{(\alpha+1,\beta+1)}}{n}
			P_n^{(\alpha,\beta)}(E_{i})
			=0 \implies 	P_n^{(\alpha,\beta)}(E_{i})=0.
		\end{align*}
		Thus $E_{i}$ is a common zero. The no common zero criterion follows immediately.
	\end{proof}
\begin{remark}
Proposition \ref{ExactCommonZeroJacobiShifted} shows that common zeros can occur only at the completion points $E_1$ and $E_2$. It is therefore useful to record the asymptotic locations of these two possible common-zero obstructions. The asymptotic location of the completion points depends on the limiting regime. First, suppose that $\alpha,\beta>-1$ are fixed and $n\to\infty$. Then
\begin{align*}
	E_1+E_2 \longrightarrow 0, \quad
	E_1E_2 \longrightarrow -\frac{1}{2} \implies E_1 \longrightarrow -\frac{1}{\sqrt{2}}, \quad
	E_2\longrightarrow \frac{1}{\sqrt{2}}.
\end{align*}
Thus, in the large-degree limit, the two possible common-zero obstructions remain in the interior of the support of the Jacobi measure. A different behavior occurs when $n$ is fixed and $\alpha+\beta\to\infty$. In this case,
\begin{align*}
	E_1+E_2 \longrightarrow 0, \quad
	E_1E_2 \longrightarrow -1 \implies E_1 \longrightarrow -1, \quad
	E_2 \longrightarrow 1.
\end{align*}
Hence, in the fixed-degree, large-parameter regime, the completion points approach the boundary points of the support of the measure.    
\end{remark}
    
    \begin{remark}\label{Symmetricevencommonzeros}
		In the ultraspherical case, namely when $\alpha=\beta$, the two completion
		points $E_{1}$ and $E_{2}$ given in \eqref{eq:E1} and \eqref{eq:E2}, respectively, are symmetric
		with respect to the origin. Indeed,
		\begin{equation*}
			E_{1}+E_2
			=
			\frac{2(n+1)(\alpha-\beta)}
			{(2n+\alpha+\beta+4)(2n+\alpha+\beta+2)} ,
		\end{equation*}
		and hence, if $\alpha=\beta$, then $E_2=-E_1$. By using the identity \cite[page 59]{Szego-OP-Book}
		\begin{equation*}
			P_{n+1}^{(\alpha+1,\alpha+1)}(-x)
			=
			(-1)^{n+1}P_{n+1}^{(\alpha+1,\alpha+1)}(x),
		\end{equation*}
		we can conclude that if one of the two completion points is a zero of
		$P_{n+1}^{(\alpha+1,\alpha+1)}$, then the other one is also a zero. Hence,
		the number of common zeros between $P_{n}^{(\alpha,\alpha)}$ and $P_{n+1}^{(\alpha+1,\alpha+1)}$ cannot be exactly one. 
		In particular, whenever common zeros occur in this case, they will be exactly two, namely 
		$E_1$ and $-E_1$. This symmetry is special to the ultraspherical case. For general Jacobi
		parameters $\alpha\neq\beta$, we have $E_1+E_2\neq0$,
		so the two completion points are not symmetric about the origin. Also, the zeros of $P_{n+1}^{(\alpha+1,\beta+1)}$ are not symmetric in general.
		Therefore, in the general Jacobi case, the number of common zeros can be at most two,
		depending on whether none, exactly one, or both completion points
		$E_1$ and $E_2$ are zeros of
		$P_{n+1}^{(\alpha+1,\beta+1)}$.
	\end{remark}

The following example shows that the polynomials $P_n^{(\alpha,\beta)}$ and $P_{n+1}^{(\alpha+1,\beta+1)}$ can have common zeros for suitable choices of the parameters $\alpha,\beta$ and the degree $n$.
	\begin{example}
		For 	$\alpha=\beta=-\frac{1}{2},$ the Jacobi polynomials $P_n^{(\alpha,\beta)}$ and $P_{n+1}^{(\alpha+1,\beta+1)}$ can be written as
		\begin{equation*}
			P_n^{(-1/2,-1/2)}(x)= 2^{1-n}T_n(x),
			\qquad
			P_{n+1}^{(1/2,1/2)}(x)=2^{-n-1}U_{n+1}(x),
		\end{equation*}
		where $T_n$ and $U_n$ are Chebyshev polynomials of the first and second kind, respectively (cf.~\cite[p.~60, Eq.~(4.1.7)]{Szego-OP-Book}).
	The corresponding trigonometric representations are
		\begin{equation*}
			T_n(\cos\theta)=\cos(n\theta),
			\qquad
			U_{n+1}(\cos\theta)=\frac{\sin((n+2)\theta)}{\sin\theta}.
		\end{equation*}
		By Proposition \ref{ExactCommonZeroJacobiShifted}, if  $P_n^{(\alpha,\beta)}$ and $P_{n+1}^{(\alpha+1,\beta+1)}$  have a common zero, then it must belong to the set
		$\{E_1,E_2\}$. Therefore, it is sufficient to check $E_1$ and $E_2$.
		For	$\alpha=\beta=-\frac{1}{2}$, the completion points are
		\begin{equation*}
			E_1=-\frac1{\sqrt2},
			\qquad
			E_2=\frac1{\sqrt2}.
		\end{equation*}
		At $\theta=\frac{\pi}{4}$, we have
		\begin{equation*}
			T_n\left(\frac1{\sqrt2}\right)	=	\cos\frac{n\pi}{4}, \qquad 	U_{n+1}\left(\frac1{\sqrt2}\right)	=	\sqrt{2}\sin\left(\frac{(n+2)\pi}{4}\right).
		\end{equation*}
		Both quantities vanish precisely when $	n\equiv 2 \pmod 4.$ As discussed in Remark \ref{Symmetricevencommonzeros},
		the same conclusion holds at $x=-1/\sqrt2$ by symmetry.	Therefore, for every $	n=4k+2, k=0,1,2,\ldots,$
		the polynomials $	P_n^{(-1/2,-1/2)}(x)$ and	$P_{n+1}^{(1/2,1/2)}(x)$ have common zeros at $\displaystyle 	E_1=-\frac1{\sqrt2},
		E_2=\frac1{\sqrt2}.$	Hence, for every $n\equiv2\pmod4$, the pair has exactly two common zeros,
		namely $E_1$ and $E_2$.
	\end{example}

 In the previous example, we showed that a common zero of the shifted Jacobi pair can occur for some choices of the parameters. We now show that this phenomenon is rare when the degree is fixed. More precisely, the set of parameters $(\alpha,\beta)$ that produce a common zero is thin in the sense that it has empty interior and two-dimensional Lebesgue measure zero in the parameter domain.
	\begin{proposition}\label{Thinparameterset-Jacobicommonzero}
		Fix $n\geq 1$, and let $	\mathcal{D}=(-1,\infty)\times(-1,\infty).$ 	Define the exceptional parameter set
		\begin{align*}
			\mathcal{A}_n	=		\left\{		(\alpha,\beta)\in\mathcal{D}:
			P_n^{(\alpha,\beta)}
			\text{ and }
			P_{n+1}^{(\alpha+1,\beta+1)}
			\text{ have a common zero}
			\right\}.
		\end{align*}
		Then
		$\mathcal{A}_n$ has empty interior and two-dimensional Lebesgue measure zero in $\mathcal{D}$.
		Consequently, the no-common-zero set $	C_n	=	\mathcal{ D}\setminus\mathcal{A}_n$ 	is open and dense in $\mathcal{D}$.
	\end{proposition}
	
	\begin{proof} By Proposition~\ref{ExactCommonZeroJacobiShifted}, the shifted Jacobi pair
		$P_n^{(\alpha,\beta)}$ and $P_{n+1}^{(\alpha+1,\beta+1)}$ has a common zero
		if and only if
		\begin{equation*}
			P_{n+1}^{(\alpha+1,\beta+1)}(E_1)
			P_{n+1}^{(\alpha+1,\beta+1)}(E_2)=0.
		\end{equation*}
		Thus,
		\begin{equation*}
			\mathcal{A}_n
			=
			\left\{
			(\alpha,\beta)\in\mathcal{D}:
			P_{n+1}^{(\alpha+1,\beta+1)}(E_1)
			P_{n+1}^{(\alpha+1,\beta+1)}(E_2)=0
			\right\}.
		\end{equation*}
		For fixed $n\geq1$, define the resultant $\mathcal{R}_n: \mathcal{D}\rightarrow \mathbb{R}$ by
		\begin{align*}
			\mathcal{R}_n(\alpha,\beta)
			=
			\operatorname{Res}_x
			\left(
			T_{n,\alpha,\beta}(x),
			P_{n+1}^{(\alpha+1,\beta+1)}(x)
			\right),
		\end{align*}
		where the quadratic polynomial $T_{n,\alpha,\beta}(x)$ is given in \eqref{eq:factored_form}, and its zeros  $E_1=E_{n,\alpha,\beta}$ and $E_2=F_{n,\alpha,\beta}$ are defined
		in \eqref{eq:E1} and \eqref{eq:E2}. Then	\begin{align*}
			\mathcal A_n
			=
			\left\{
			(\alpha,\beta)\in\mathcal{D}:
			\mathcal{R}_n(\alpha,\beta)=0
			\right\}.
		\end{align*}
		Next, we show that $\mathcal{R}_n$ is not identically zero on $\mathcal{D}$. For this,	it is enough to exhibit one admissible pair $(\alpha,\beta)$ for which
		$\mathcal{R}_n(\alpha,\beta)\neq0$. 	For $\alpha=\beta=t$, we have $	E_2=-E_1,$
		and, from \eqref{eq:E1}-\eqref{eq:E2},
		\begin{align*}
			E_1^2
			=
			\frac{n+2t+2}{2n+2t+3}.
		\end{align*}
		Thus, for fixed $n,~ E_1\longrightarrow -1,	E_2\longrightarrow 1$,
		as $t\to\infty.$	For fixed degree, the zeros of Gegenbauer polynomials tend to the origin
		as the parameter tends to infinity \cite[page 169]{Temme-Specialfunctions1996}. Since
		$P_{n+1}^{(t+1,t+1)}$ is a nonzero constant multiple of Gegenbauer polynomial
		$C_{n+1}^{\,t+3/2}$, it follows that the zeros of
		$P_{n+1}^{(t+1,t+1)}$ tend to $0$ as $t\to\infty$. 
			Thus, for sufficiently large $t$, all zeros of
		$P_{n+1}^{(t+1,t+1)}$ lie in a small neighborhood of the origin, while
		$E_1$ and $E_2$ are close to $-1$ and $1$, respectively. Hence neither
		$E_1$ nor $E_2$ is a zero of $P_{n+1}^{(t+1,t+1)}$, and therefore
		\begin{align*}
			P_{n+1}^{(t+1,t+1)}(E_1)
			P_{n+1}^{(t+1,t+1)}(E_2)
			\neq0.
		\end{align*}
		Since $T_{n,t,t}$ is monic with roots $E_1$ and $E_2$, we have
		\begin{align*}
			\mathcal{R}_n(t,t)
			=
			P_{n+1}^{(t+1,t+1)}(E_1)
			P_{n+1}^{(t+1,t+1)}(E_2)
			\neq0,
		\end{align*}
		for sufficiently large $t$.
		Hence, $\mathcal{R}_n$ is not identically zero on
		$\mathcal{D}$.
		
	Next, we see that	the coefficients of $T_{n,\alpha,\beta}$ and
		$P_{n+1}^{(\alpha+1,\beta+1)}$ are real-analytic functions of
		$(\alpha,\beta)$ on $\mathcal{D}$, and hence $\mathcal{R}_n$ is real analytic as well.
		The zero set of a nontrivial real-analytic function cannot contain a
		nonempty open subset of its domain. Therefore $\mathcal{A}_n$ has empty
		interior in $\mathcal{D}$. Because $\mathcal{D}$ is open and connected in $\mathbb{R}^2$ and $\mathcal{R}_n$ is not identically zero on
		$\mathcal{D}$,  \cite[Proposition 1]{Mityagin2020-Zerosetofrealanalytic} implies that the zero set of a nontrivial
		real-analytic function $\mathcal{R}_n$ has two-dimensional Lebesgue measure zero.
		Further, the continuity of $\mathcal{R}_n$ implies that $\mathcal{A}_n=\mathcal{R}_n^{-1}(\{0\})$
		is closed relative to $\mathcal D$. Hence $	C_n	=	\mathcal{ D}\setminus\mathcal{A}_n$	is open in $\mathcal{D}$. Since $\mathcal{A}_n$ has empty interior,
		$C_n$ is dense in $\mathcal{D}$.
	\end{proof}
	
	\begin{remark}
	Although Proposition \ref{Thinparameterset-Jacobicommonzero} is stated for a fixed degree, it also yields
	a simultaneous generic statement over all degrees. Indeed, if
	\begin{align*}
		\mathcal{A}
		=
		\bigcup_{n\geq1}\mathcal{A}_n,
	\end{align*}
then $\mathcal{A}$ has two-dimensional Lebesgue measure zero in
$\mathcal{D}$ and is of first category. Consequently, for every $(\alpha,\beta)\in\mathcal{D}\setminus\mathcal{A}$, the Jacobi polynomials $P_n^{(\alpha,\beta)}$ and $P_{n+1}^{(\alpha+1,\beta+1)}$ have no common zero for every $n\geq 1$.  
\end{remark}

\subsection[
Cross-family Turan-type inequality for the doubly shifted pair
Pn(alpha,beta) and P(n+1)(alpha+1,beta+1)
]{
	Cross-family Tur\'an-type inequality for the doubly shifted pair
	$P_n^{(\alpha,\beta)}$ and
	$P_{n+1}^{(\alpha+1,\beta+1)}$
}\label{Turan inequalities}

We now specialize the preceding result to monic Jacobi polynomials. For the monic Jacobi polynomial $P^{(\alpha,\beta)}_n$, the parameter shift $(\alpha,\beta,n)\mapsto(\alpha+1,\beta+1,n+1)$ does not preserve the interlacing property between the zeros of $P_n^{(\alpha,\beta)}$ and $P_{n+1}^{(\alpha+1,\beta+1)}$. However, Corollary \ref{Interlace-Jacobi-both-shifted-degree-parameter} explicitly identifies two completion points $E_1$ and $E_2$, which restore this interlacing. Subsequently, we apply Theorem \ref{thm:general} to this Jacobi case to obtain the following cross-family Tur\'an-type inequality  and its sharpened form. The inequality involves neighboring degrees $n-1$, $n$, $n+1$, as in the classical Tur\'an expression, but with the Jacobi parameters shifted contiguously through the levels $(\alpha,\beta)$, $(\alpha+1,\beta+1)$ and  $(\alpha+2,\beta+2)$.

\begin{theorem}\label{thm:JacobiTuran}
	Let $\alpha,\beta>-1$, $n\ge1$, and let $E_1,E_2\in(-1,1)$
	be the completion points given by \eqref{eq:E1} and \eqref{eq:E2}. Assume that $P_n^{(\alpha,\beta)}$ and $P_{n+1}^{(\alpha+1,\beta+1)}$  have no common zeros. Then, for every $x\in \mathbb{R}$, the following   inequality  holds
		\begin{align}\label{eq:JacobiTuranSharp}
		\nonumber	(x-E_1)(x-E_2)\Bigl[
		(n+1)P_n^{(\alpha,\beta)}P_n^{(\alpha+2,\beta+2)}
		&-n\,P_{n-1}^{(\alpha+1,\beta+1)}P_{n+1}^{(\alpha+1,\beta+1)}
		\Bigr]
		\\&<(2x-E_1-E_2)\,P_n^{(\alpha,\beta)}P_{n+1}^{(\alpha+1,\beta+1)}-\left(P_{n+1}^{(\alpha+1,\beta+1)}\right)^2.
	\end{align}
In particular, this yields the following cross-family Tur\'an-type
	inequality for the doubly shifted pair
	$P_n^{(\alpha,\beta)}$ and
	$P_{n+1}^{(\alpha+1,\beta+1)}$:
	\begin{align}\label{eq:JacobiTuran}
		(x-E_1)(x-E_2)\Bigl[
		(n+1)P_n^{(\alpha,\beta)}P_n^{(\alpha+2,\beta+2)}
		-n\,P_{n-1}^{(\alpha+1,\beta+1)}P_{n+1}^{(\alpha+1,\beta+1)}
		\Bigr]
		<(2x-E_1-E_2)\,P_n^{(\alpha,\beta)}P_{n+1}^{(\alpha+1,\beta+1)}.
	\end{align}
\end{theorem}

\begin{proof}	Let us define  
	$G_{n+1}(x) := {P}_{n+1}^{(\alpha+1,\beta+1)}(x)$
	and $P_{n}(x) := {P}_{n}^{(\alpha,\beta)}(x)$.  
	For any $\alpha,\beta>-1$
	and $n\ge1$, the derivative of monic Jacobi polynomials is 
	\begin{equation}\label{eq0:Jacobider}
		\frac{d}{dx}P_n^{(\alpha,\beta)}(x)
		=n\,P_{n-1}^{(\alpha+1,\beta+1)}(x).
	\end{equation}	
	Now,  using these derivative identities, we can write the Wronskian as
	\begin{align}\label{eq1: Wdoublyshifted}
		\nonumber	W(P^{(\alpha,\beta)}_n,P_{n+1}^{(\alpha+1,\beta+1)})(x)&=P^{(\alpha,\beta)}_{n}(x)\frac{d}{dx}P_{n+1}^{(\alpha+1,\beta+1)}(x)-P^{(\alpha+1,\beta+1)}_{n+1}(x)\frac{d}{dx}P_n^{(\alpha,\beta)}(x)\\
		&=(n+1)P_n^{(\alpha,\beta)}(x)P_n^{(\alpha+2,\beta+2)}(x)
		-nP_{n-1}^{(\alpha+1,\beta+1)}(x)P_{n+1}^{(\alpha+1,\beta+1)}(x).
	\end{align}
	From Corollary \ref{Interlace-Jacobi-both-shifted-degree-parameter},  $(x - E_1)(x - E_2) P^{(\alpha, \beta)}_n(x)\prec P^{(\alpha+1, \beta+1)}_{n+1}(x)$, where $E_1$ and $E_2$ are given in \eqref{eq:E1} and \eqref{eq:E2}. Thus, all the hypotheses of Theorem \ref{thm:general} are satisfied, and the result follows by substituting the value of	$W(P^{(\alpha,\beta)}_n,P_{n+1}^{(\alpha+1,\beta+1)})(x)$ in Theorem \ref{thm:general}.  
\end{proof}

Applying Proposition~\ref{lem:generaldecomposition} to the Jacobi polynomials $P_n^{(\alpha,\beta)}$ and $G_{n+1}=P_{n+1}^{(\alpha+1,\beta+1)}$  yields the following Wronskian decomposition and estimates.

\begin{corollary}\label{thm:Jacobi-Wronskian-combined}
	Let $\alpha,\beta>-1$, $n\geq 1$, and let $E_1<E_2$ be as in \eqref{eq:E1} and \eqref{eq:E2}. Assume that $P_n^{(\alpha,\beta)}$ and $P_{n+1}^{(\alpha+1,\beta+1)}$  have no common zeros. Let \[C_{n,\alpha,\beta}\!=\!\frac{(n+\alpha+\beta+1)\lambda_{n+1}^{(\alpha+1,\beta+1)}}{n}~\text{and}~ \tilde{R}(x)=\!	-\,(P_{n+1}^{(\alpha+1,\beta+1)}(x))^2
	+(2x-E_1-E_2)P_{n}^{(\alpha+1,\beta+1)}(x)P_{n+1}^{(\alpha+1,\beta+1)}(x).\] 
	 Then, for
	every $x\in\mathbb{R}$,
	\begin{equation}\label{eq:Jacobi Wronkian decomposition}	C_{n,\alpha,\beta}W(P^{(\alpha,\beta)}_n,P_{n+1}^{(\alpha+1,\beta+1)})(x)
		=	-(x-E_1)(x-E_2)W(P_{n}^{(\alpha+1,\beta+1)},P_{n+1}^{(\alpha+1,\beta+1)})(x)
		+\tilde{R}(x).\end{equation}
Further
	\begin{equation*}
C_{n,\alpha,\beta}W(P^{(\alpha,\beta)}_n,P_{n+1}^{(\alpha+1,\beta+1)})(x)
	\begin{cases}
		>\tilde{R}(x) & x\in(E_1,E_2),\\
		<\tilde{R}(x) & x\in (-\infty,E_1)\cup(E_2,\infty),
	\end{cases}
	\end{equation*}
	with equality if and only if $x\in\{E_1,E_2\}$.
		Moreover, for every $x\in\mathbb{R}$,
	\begin{align*}
	C_{n,\alpha,\beta}(2x-E_1-E_2)&P_{n}^{(\alpha,\beta)}(x)P_{n+1}^{(\alpha+1,\beta+1)}(x)
		\\&>(x-E_1)(x-E_2)\tilde{R}(x)-(x-E_1)^2(x-E_2)^2W(P_{n}^{(\alpha+1,\beta+1)},P_{n+1}^{(\alpha+1,\beta+1)})(x).
	\end{align*}
\end{corollary}

\begin{proof}
	Recall that the relation \eqref{Mixed-TTRR-Monic-Jacobi2} has the form \eqref{GeneralMixedTTRR3} with 
	$G_{n+1}(x) := {P}_{n+1}^{(\alpha+1,\beta+1)}(x)$,   $P_{n}(x) := {P}_{n}^{(\alpha,\beta)}(x)$,  
	 $Q_{n}(x)=G_n(x):= {P}_{n}^{(\alpha+1,\beta+1)}(x)$, $A(x) := C_{n,\alpha,\beta}>0$  
	 and  
	 $B(x) := x - \frac{\alpha - \beta}{2n + \alpha + \beta + 2}$.  
	Since $A'=0$ and $B'=1$, Proposition \ref{lem:generaldecomposition} gives \eqref{eq:Jacobi Wronkian decomposition}. The piecewise inequality follows from the positivity of  $W\left(P_n^{(\alpha+1,\beta+1)},P_{n+1}^{(\alpha+1,\beta+1)}\right)(x)$ for $x\in\mathbb{R}.$ The last inequality then follows by
	inserting the above choices of $G_{n+1}$, $Q_n$,
	$P_n$, $A$, and $B$ into \eqref{eq:pos-bound}.

\end{proof}
 
\subsection{Stieltjes representation and complete monotonicity}\label{Stieltjes}

The partial fraction representation obtained in the proof of  Lemma \ref{lem:Wronskiangeneral}, which follows from the completed
interlacing established in Theorem \ref{MainTheorem3}, also provides an analytic interpretation of the augmented Jacobi quotient.

\begin{corollary}	\label{cor:JacobiStieltjesCM}
	Let $\alpha,\beta>-1$, $n\ge1$, and let $E_1,E_2\in(-1,1)$
	be the completion points given by \eqref{eq:E1} and \eqref{eq:E2}. Assume that $P_n^{(\alpha,\beta)}$ and $P_{n+1}^{(\alpha+1,\beta+1)}$  have no common zeros. Let $	z_1<z_2<\cdots<z_{n+2}$	be the zeros of $	(z-E_1)(z-E_2)P_n^{(\alpha,\beta)}(z).$	Then, for $z\in \mathbb{C}\setminus\{z_1,z_2,...,z_{n+2}\}$, the quotient 	admits the Stieltjes representation
	\begin{equation}	\label{eq:JacobiStieltjesRepresentation}
		\mathcal R_n(z)	:=	\frac{	P_{n+1}^{(\alpha+1,\beta+1)}(z)	}{
			(z-E_1)(z-E_2)P_n^{(\alpha,\beta)}(z)}=	\int_{\mathbb R}\frac{d\mu_n(t)}{z-t},
	\end{equation}
	where $	\mu_n	:=
	\sum_{j=1}^{n+2}R_j\delta_{z_j}$
	is a  probability measure, with
	\begin{equation}	\label{eq:JacobiStieltjesWeights}
		R_j	=	\frac{	P_{n+1}^{(\alpha+1,\beta+1)}(z_j)}{	\left.\displaystyle\frac{d}{dz}
			\left[	(z-E_1)(z-E_2)P_n^{(\alpha,\beta)}(z)\right]
			\right|_{z=z_j}	}>0,
		\qquad
		j=1,\ldots,n+2.
	\end{equation}
	Moreover, $\mathcal{R}_n(x)$ is completely monotone on $(z_{n+2},\infty)$.
	In particular, the quotient $\mathcal{R}_n$ is completely monotone on $(1,\infty)$.
\end{corollary}

\begin{proof} We proved in Corollary \ref{Interlace-Jacobi-both-shifted-degree-parameter} that  $(x - E_1)(x - E_2) P^{(\alpha, \beta)}_n(x)\prec P^{(\alpha+1, \beta+1)}_{n+1}(x)$, where $E_1$ and $E_2$ are given in \eqref{eq:E1} and \eqref{eq:E2}. 
	Now, by applying the partial fraction decomposition obtained in the proof of
	Lemma~\ref{lem:Wronskiangeneral} to the Jacobi pair
	\begin{equation*}
		 P_n(x)=P_n^{(\alpha,\beta)}(x),
		\qquad
		 G_{n+1}(x)=P_{n+1}^{(\alpha+1,\beta+1)}(x),
	\end{equation*}
	we obtain
	\begin{equation}\label{Rn:PartialfractionJacobiquotient}
		\mathcal R_n(z):=	\frac{	P_{n+1}^{(\alpha+1,\beta+1)}(z)	}{
			(z-E_1)(z-E_2)P_n^{(\alpha,\beta)}(z)}	=
		\sum_{j=1}^{n+2}\frac{R_j}{z-z_j},
		\qquad
		R_j>0.
	\end{equation}
	This proves \eqref{eq:JacobiStieltjesRepresentation}, where $	\mu_n
	:=
	\sum_{j=1}^{n+2}R_j\delta_{z_j}$ is a probability measure. Note that the verification that $\sum_{j=1}^{n+2} R_j = 1$ follows by the same argument used to prove the corresponding summation identity in Theorem \ref{thm:general}.
	Now,
	differentiating  \eqref{Rn:PartialfractionJacobiquotient}, $m$ times with respect to $x$ gives
	\begin{equation*}
		\mathcal R_n^{(m)}(x)	=	(-1)^m m!	\sum_{j=1}^{n+2}
		\frac{R_j}{(x-z_j)^{m+1}} \implies	(-1)^m\mathcal R_n^{(m)}(x)
		=
		m!
		\sum_{j=1}^{n+2}
		\frac{R_j}{(x-z_j)^{m+1}}
		>0,~ \text{for} ~x\in (z_{n+2},\infty), 
	\end{equation*}
	This proves that  $\mathcal{R}_n(x)$ is completely monotone on $(z_{n+2},\infty)$. Since $z_{n+2}<1$,	the conclusion on $(1,\infty)$ follows immediately.
\end{proof}

\begin{remark}
	We can also translate the  partial fraction representation
	\eqref{Rn:PartialfractionJacobiquotient}  to the
	half-line. Indeed, if
	\begin{equation*}
		\widetilde{\mathcal{R}}_n(s)
		:=
		\mathcal{R}_n(z_{n+2}+s),
		\qquad s>0,
	\end{equation*}
	then
	\begin{equation*}
		\widetilde{\mathcal{R}}_n(s)
		=
		\frac{R_{n+2}}{s}
		+
		\sum_{j=1}^{n+1}
		\frac{R_j}{s+(z_{n+2}-z_j)}.
	\end{equation*}
	Since $R_j>0$ and $z_{n+2}-z_j>0$ for $j=1,\ldots,n+1$,
	$\widetilde{\mathcal{R}}_n$ is a Stieltjes function on $(0,\infty)$.
	For background on Stieltjes functions and their real-variable
	characterizations, see~\cite{SokalExpomath2010}.
\end{remark}
Furthermore, the complete monotonicity of the quotient $\mathcal{R}_n$ yields higher-order derivative inequalities and the following geometric properties. For every integer $m\geq0$ and every $x>z_{n+2}$,	define
\begin{equation*}
	M_m(x)
	:=
	\sum_{j=1}^{n+2}
	\frac{R_j}{(x-z_j)^{m+1}}. \end{equation*}
    Then $ \displaystyle{M_m(x)
	=
	\frac{(-1)^m}{m!}\mathcal R_n^{(m)}(x)} $. Since $R_j>0$ and $x-z_j>0$, the Cauchy-Schwarz inequality gives
\begin{equation*}
	M_m(x)M_{m+2}(x)
	>
	M_{m+1}(x)^2.
\end{equation*} Indeed, equality would require the quantities $\displaystyle\frac{\sqrt{R_j}}{(x-z_j)^{(m+1)/2}}$ and $\displaystyle\frac{\sqrt{R_j}}{(x-z_j)^{(m+3)/2}}$ to be proportional in $j$, which would force $x-z_j$ to be independent of $j$ and this would contradict the fact that the $z_j$ are distinct. Hence,	we obtain 	\begin{equation}
	\label{eq:JacobiHigherDerivativeInequality}
	\mathcal R_n^{(m)}(x)\mathcal R_n^{(m+2)}(x)
	>
	\frac{m+2}{m+1}
	\left(\mathcal R_n^{(m+1)}(x)\right)^2.
\end{equation}
In particular, for
$m=0$,
\begin{equation}\label{logconvex-Rn}
	\mathcal R_n(x)\mathcal R_n''(x)
	>
	2\left(\mathcal R_n'(x)\right)^2>
	\left(\mathcal R_n'(x)\right)^2,
	\qquad
	x>z_{n+2}.
\end{equation}
So $\mathcal R_n$ is strictly log-convex on $(z_{n+2},\infty)$. Further, by \eqref{logconvex-Rn}, we easily see that the reciprocal function $(\mathcal{R}_n(x))^{-1}$ is strictly concave on $(z_{n+2},\infty)$.

	\section{Meixner-Pollaczek polynomials}\label{Meixner-Pollaczek polynomials}
	
	The monic Meixner-Pollaczek polynomials \cite[Section 9.7]{Koekoek-Book-HypergeometricOP}, denoted by $P_n^{(\lambda)}(x; \phi)$, are defined by
	\begin{align}
		P_n^{(\lambda)}(x;\phi) &= i^n (2\lambda)_n \left( \frac{e^{2i\phi}}{e^{2i\phi} - 1} \right)^n \,_2F_1\left( -n, \lambda + i x; 2\lambda; 1 - \frac{1}{e^{2i\phi}} \right)
	\end{align}
	
	These polynomials are orthogonal with respect  to $d\mu(x) = |\Gamma(\lambda + i x)|^2 e^{(2\phi - \pi)x}dx$ on  $(-\infty, \infty)$, where $n \in \mathbb{N}$, $\lambda > 0$, and $0 < \phi < \pi$. Here, $\Gamma(z)$ represents the Gamma function
	\begin{align*}
		\Gamma(z) &= \int_0^\infty t^{z-1} e^{-t} \, dt, \qquad \Re(z)>0
	\end{align*} and  the Gauss hypergeometric series $_2F_1 \left( a, b; c; z \right)$ is defined as
	\begin{align}\label{Gauss Hypergeometric series}
		_2F_1\left( a, b; c; z \right)&= \sum_{k=0}^{\infty} \frac{(a)_k (b)_k}{(c)_k} \frac{z^k}{k!}.
	\end{align}
	The Pochhammer symbol $(a)_n$ is defined as follows:
	\begin{align*}
		(a)_0=1, \qquad (a)_n &= a(a+1)(a+2)\cdots(a+n-1)\quad \text{for}~ n \geq 1.
	\end{align*}
	
	\begin{lemma}
		Let $n \in \mathbb{N}$, $\lambda > 0$, and $\phi \in (0,\pi)$. Then the polynomial
		\begin{align}\label{Quadraticterm-involve-interlacing}
			T_{n,\lambda,\phi}(x)=(1-\cos2\phi)x^2+(n+1)\sin(2\phi) x+ \lambda^2-(\lambda^2+n\lambda+\lambda)\cos2\phi,
		\end{align}
		has two distinct  real zeros if and only if
		\begin{align}\label{condition-phi-positive-discriminant}
			\phi \in \left(0, \frac{1}{2}\arccos\left(\frac{2\lambda-n-1}{2\lambda+n+1}\right)\right)\cup \left(\pi-\frac{1}{2}\arccos\left(\frac{2\lambda-n-1}{2\lambda+n+1}\right), \pi \right).
		\end{align}
	\end{lemma}
	
	\begin{proof}
		Since $\phi\in(0,\pi)$, the coefficient $1-\cos 2\phi$ of $x^2$ is non-zero. Hence $T_{n,\lambda,\phi}(x)$ is quadratic in $x$ and its zeros are real and distinct precisely when its discriminant is positive.  Denoting by
		$D := \mathrm{Disc}_x\!\left(T_{n,\lambda,\phi}\right)$
		the discriminant of $T_{n,\lambda,\phi}(x)$, we have
		\begin{align}\label{Discriminant}
			D=(n+1)^2\sin^2(2\phi)-4(1-\cos2\phi)(\lambda^2-(\lambda^2+n\lambda+\lambda)\cos2\phi).
		\end{align}
		The zeros of $T_{n,\lambda,\phi}(x)$ are 
		\begin{align*}
			x=\frac{-(n+1)\sin2\phi \pm \sqrt{D}}{2(1-\cos2\phi)}.
		\end{align*}
		To simplify \eqref{Discriminant}, we denote $\cos(2\phi)=y, n+1=A$ and $B=\lambda^2+n \lambda+\lambda=\lambda^2+A\lambda$. We can write \eqref{Discriminant} as
		\begin{align}
			\nonumber	D&=A^2(1-y^2)-4(1-y)(\lambda^2-By)\\
			\nonumber	&=-(A+2\lambda)^2y^2+4\lambda(A+2\lambda)y+A^2-4\lambda^2\\
			\nonumber	&=-(A+2\lambda)f(y)
		\end{align}
        where $f(y):= (A+2\lambda)y^2-4\lambda y-(A-2\lambda)$.	 Then $D>0$ is equivalent to $f(y)<0$. The polynomial $f(y)$ vanishes at
		\begin{align*}
			y &= \frac{4\lambda \pm \sqrt{(4\lambda)^2 + 4(A^2 - 4\lambda^2)}}{2(A+2\lambda)}
			= \frac{2\lambda \pm A}{A+2\lambda}.
		\end{align*}
		Thus, the zeros of $f(y)$ are $y=1$ and
		$y=1-\frac{2A}{A+2\lambda}<1$.
		Since $A+2\lambda>0$, the quadratic $f(y)$ opens upward, and hence
		$f(y)<0$ for $	y \in \left(\frac{2\lambda-A}{A+2\lambda},\,1\right).$
		Consequently, $D>0$ if and only if
		\begin{align*}
			\frac{2\lambda-n-1}{2\lambda+n+1}
			< \cos 2\phi < 1.
		\end{align*}
		Since  $\arccos$ is decreasing on $(-1,1)$, it follows that
		\begin{align*}
			0 < \arccos(\cos 2\phi)
			< \arccos\!\left(\frac{2\lambda-n-1}{2\lambda+n+1}\right).
		\end{align*}
		Using the identity
		\begin{align*}
			\arccos(\cos2\phi) =
			\begin{cases}
				2\phi, & \text{if } \phi \in (0,\frac{\pi}{2}) \\
				2\pi - 2\phi, & \text{if } \phi \in (\frac{\pi}{2}, \pi)
			\end{cases},
		\end{align*} we obtain  \eqref{condition-phi-positive-discriminant}. Also, the polynomial $T_{n,\lambda,\phi}(x)$ can be factorized as
		\begin{align}\label{factorised-Quadraticterm-involve-interlacing}
			T_{n,\lambda,\phi}(x)=(1-\cos2\phi)(x-E_1)(x-E_2),
		\end{align}where 
		\begin{align}\label{Factor1-responsible-Interlacing}
			E_1:=E_{n,\lambda,\phi}=\frac{-(n+1)\sin2\phi \,- (n+1+2\lambda)\sqrt{(1-\cos2\phi)\left(\cos2\phi-\frac{2\lambda-n-1}{2\lambda+n+1}\right)}}{2(1-\cos2\phi)},
		\end{align} and 
		\begin{align}\label{Factor2-responsible-Interlacing}
			E_2:=	F_{n,\lambda,\phi}=\frac{-(n+1)\sin2\phi \,+ (n+1+2\lambda)\sqrt{(1-\cos2\phi)\left(\cos2\phi-\frac{2\lambda-n-1}{2\lambda+n+1}\right)}}{2(1-\cos2\phi)}
		\end{align} with $E_1<E_2$.
		This completes the proof. 
	\end{proof}
	
	It is well known that, for any $n \in \mathbb{N}$, $\lambda > 0$, and
	$\phi \in (0,\pi)$, $P_{n+1}^{(\lambda)}(x;\phi)\prec P_{n}^{(\lambda)}(x;\phi)$. In contrast, such an
	interlacing property does not generally hold for the zeros of
	$P_{n}^{(\lambda)}(x;\phi)$ and $P_{n+1}^{(\lambda+1)}(x;\phi)$.
	Nevertheless, it was shown in \cite[Lemma~3.3]{Jooste-Jordaan-Numeralg-2025}
	that these zeros interlace in the special case $\phi = \pi/2$. An interesting open problem is whether suitable conditions on the
	parameters can be imposed to obtain either full or partial interlacing
	between the zeros of $P_{n}^{(\lambda)}(x;\phi)$ and
	$P_{n+1}^{(\lambda+1)}(x;\phi)$. In the following result, we demonstrate
	that partial interlacing can indeed be achieved by restricting $\phi$ to
	a range depending on $n \in \mathbb{N}$ and $\lambda > 0$. A key role in
	this analysis is played by the location of the quantities
	$E_1 := E_{n,\lambda,\phi}$ and $E_2 := F_{n,\lambda,\phi}$, defined in
	\eqref{Factor1-responsible-Interlacing} and
	\eqref{Factor2-responsible-Interlacing}, respectively.

	\begin{corollary}\label{MP-Interlacing-two-points}
		Let $n \in \mathbb{N}$, $\lambda > 0$, and $\phi$ satisfy \eqref{condition-phi-positive-discriminant}.
		Let $\{z_{k,n}\}_{k=1}^{n}$ and $\{y_{k,n+1}\}_{k=1}^{n+1}$ denote the zeros of the polynomials $P_n^{(\lambda)}(x;\phi)$, and $P_{n+1}^{(\lambda+1)}(x;\phi)$,  respectively. Assume that $P_{n}^{(\lambda)}(x;\phi)$ and $P_{n+1}^{(\lambda+1)}(x;\phi)$ have no common zero. Then the configuration $E_1<y_{1,n+1}$ and $E_2>y_{n+1,n+1}$ is impossible, where  $E_1$ and $E_2$ are given in \eqref{Factor1-responsible-Interlacing} and \eqref{Factor2-responsible-Interlacing}. Also,  $P_{n+1}^{(\lambda+1)}(x;\phi)\prec P_{n}^{(\lambda)}(x;\phi)$  whenever $E_1$ and $E_2$ lie simultaneously in one of the following intervals: $(-\infty,y_{1,n+1})$, $(y_{k,n+1},y_{k+1,n+1})$ for some fixed $k$, or $(y_{n+1,n+1},\infty)$. 	Moreover, the following interlacing properties hold:
		
		\begin{enumerate}[label=\alph*)]
			\item 
			If $E_1<y_{1,n+1}$ and $y_{k,n+1}<E_2<y_{k+1,n+1}$ for some $k\in\{1,\dots,n\}$, then  $(x-E_2)P_{n}^{(\lambda)}(x;\phi)\prec P_{n+1}^{(\lambda+1)}(x;\phi)$.
			
			\item 
			If $E_2>y_{n+1,n+1}$ and $y_{k,n+1}<E_1<y_{k+1,n+1}$ for some $k\in\{1,\dots,n\}$, then  $P_{n+1}^{(\lambda+1)}(x;\phi) \prec (x-E_1)P_{n}^{(\lambda)}(x;\phi)$.
			
			\item 
			If there exist distinct $k,k'\in\{1,\dots,n\}$ such that $y_{k,n+1} < E_1 < y_{k+1,n+1}$ and
			$y_{k',n+1} < E_2 < y_{k'+1,n+1},
			$  then, for $l\in\{1,\dots,n-1\}$, \begin{itemize}\item[(i)]   $(x-E_1)P_{n+1}^{(\lambda+1)}(x;\phi)\prec (x-E_2)P_{n}^{(\lambda)}(x;\phi)$ whenever $
				y_{k,n+1} < z_{l,n} < E_1 < z_{l+1,n} < y_{k+1,n+1}
				$,
				\item[(ii)] $(x-E_2)P_{n+1}^{(\lambda+1)}(x;\phi)\prec (x-E_1)P_{n}^{(\lambda)}(x;\phi)$  whenever 
				$y_{k',n+1} < z_{l,n} < E_2 < z_{l+1,n} < y_{k'+1,n+1}$,  \end{itemize}
		\end{enumerate}
	\end{corollary}
	
	\begin{proof}
		Monic Meixner-Pollaczek polynomials satisfy the following mixed three-term recurrence relation (cf. ~\cite[eq. ~(3.3)]{Jooste-Jordaan-Numeralg-2025})
		\begin{align}
			P_{n}^{(\lambda)}(x;\phi) =  - \frac{4 \sin^2 \phi \,(x - \lambda \cot \phi)}{(2 \lambda + n)(2 \lambda + n + 1)}\, P_{n+1}^{(\lambda+1)}(x;\phi)+
			\frac{2T_{n,\lambda,\phi}(x)}{(2 \lambda + n)(2 \lambda + n + 1)}\, P_{n}^{(\lambda+1)}(x;\phi),
		\end{align}
		where $T_{n,\lambda,\phi}(x)$ is given in \eqref{Quadraticterm-involve-interlacing}. Now using \eqref{factorised-Quadraticterm-involve-interlacing}, we can rewrite the above equation as 
		\begin{align}
			\nonumber	\frac{(2 \lambda + n)(2 \lambda + n + 1)}{2(1-\cos2\phi)}	P_{n}^{(\lambda)}(x;\phi) &=  - \frac{2 \sin^2 \phi \,(x - \lambda \cot \phi)}{(1-\cos2\phi)}\, P_{n+1}^{(\lambda+1)}(x;\phi)\\
			&\hspace{3cm}+ (x-E_1)(x-E_2)
			P_{n}^{(\lambda+1)}(x;\phi),
		\end{align}
		where $E_1$ and $E_2$ are given in \eqref{Factor1-responsible-Interlacing} and \eqref{Factor2-responsible-Interlacing}. Define
		$G_{n+1}(x) := P_{n+1}^{(\lambda+1)}(x;\phi)$,
		$Q_{n}(x) := P_{n}^{(\lambda+1)}(x;\phi)$,
		and
		$P_{n}(x) := P_{n}^{(\lambda)}(x;\phi)$.
		Further, set
		$A(x) := \frac{(2 \lambda + n)(2 \lambda + n + 1)}{2(1-\cos 2\phi)} > 0$
		and
		$B(x) := -\frac{2 \sin^2 \phi \,(x - \lambda \cot \phi)}{1-\cos 2\phi}$. The range of $\phi$ ensures that the quantities $E_1$ and $E_2$, defined in
		\eqref{Factor1-responsible-Interlacing} and
		\eqref{Factor2-responsible-Interlacing}, are real numbers.
		It follows from \eqref{factorised-Quadraticterm-involve-interlacing} that
		$B(E_1) \neq 0$ and $B(E_2) \neq 0$, since $B(x)$ vanishes only at $\lambda \cot\phi$ and a direct calculation shows that $T_{n,\lambda,\phi}(\lambda\cot\phi)=\lambda(n+1+2\lambda)>0$.
		Moreover, for any $n \in \mathbb{N}$, $\lambda > 0$, and $\phi \in (0,\pi)$,
		the zeros of $P_{n}^{(\lambda+1)}(x;\phi)$ and $P_{n+1}^{(\lambda+1)}(x;\phi)$
		interlace. Consequently, all the hypotheses of Theorem~\ref{MainTheorem4}
		are satisfied, and the result follows directly from Theorem~\ref{MainTheorem4}.
	\end{proof}

	\begin{table}[ht]
		\centering
		\caption{Zeros of $P_{n}^{(\lambda)}(x;\phi)$ (denoted by $z_{k,n}$) and $P_{n+1}^{(\lambda+1)}(x;\phi)$ (denoted by $y_{k,n+1}$), where $E_1$ and $E_2$ are defined in  \eqref{Factor1-responsible-Interlacing} and \eqref{Factor2-responsible-Interlacing}, respectively, and $\phi$ satisfies \eqref{condition-phi-positive-discriminant}. }
		\label{tab2:Interlacingzeros1_MeixnerPoll}
		\footnotesize 
		\setlength{\tabcolsep}{12pt} 
		\begin{tabular}{@{}c cc | cc@{}}
			\toprule
			
			\addlinespace[2pt]
			& \multicolumn{2}{c|}{$n=6, \lambda=0.12, \phi=\frac{7\pi}{9}$} & \multicolumn{2}{c}{$n=7, \lambda=2, \phi=\frac{\pi}{4}$} \\
			& \multicolumn{2}{c|}{$E_1 = -0.019388, E_2=8.36166$} & \multicolumn{2}{c}{$E_1=-7.4641, E_2=-0.535898$} \\
			\midrule
			$k$ & $z_{k,6}$ & $y_{k,7}$ & $z_{k,7}$ & $y_{k,8}$ \\
			\midrule
			
			1 & $-0.495851$ & $-0.94902$ & $-14.3009$ & $-18.3133$ \\
			2 & $0.0628661$ & $0.28911$ & $-9.62676$ & $-13.1392$ \\
			3 & $0.995837$ & $1.6211$ & $-6.32237$ & $-9.37474$ \\
			4 & $2.81056$ & $3.48124$ & $-3.81923$ & $-6.41755$ \\
			5 & $5.56002$ & $6.0181$ & $-1.88168$ & $-4.0225$ \\
			6 & $9.80094$ & $9.4672$ & $-0.317583$ & $-2.03526$ \\
			7 & $-$ & $14.4424$ & $1.26853$ & $-0.281821$ \\
			8 & $-$ & $-$ & $-$ & $1.5844$ \\
			\midrule
			& \multicolumn{2}{c|}{$(x-E_1)P_{n+1}^{(\lambda+1)}(x;\phi)\prec (x-E_2) P_{n}^{(\lambda)}(x;\phi)$} & \multicolumn{2}{c}{$(x-E_2)P_{n+1}^{(\lambda+1)}(x;\phi)\prec (x-E_1) P_{n}^{(\lambda)}(x;\phi)$} \\
			\bottomrule
			
		\end{tabular}
		
		\vspace{0.2cm}
		\begin{minipage}{0.9\textwidth}
			\footnotesize
			\textbf{Note:} Observe that $y_{1,7}<z_{1,6}<E_1 < z_{2,6}<y_{2,7}$ and $y_{5,7}<E_2 < y_{6,7}$ for $n=6, \lambda=0.12, \phi=\frac{7\pi}{9}$; similarly, $y_{3,8}<E_1<y_{4,8}$ and $y_{6,8}<z_{5,7}<E_2 < z_{6,7}<y_{7,8}$ for $n=7, \lambda=2, \phi=\frac{\pi}{4}$.
		\end{minipage}
	\end{table}	

    \begin{table}[H]
		\centering
		\caption{Zeros of $P_{n}^{(\lambda)}(x;\phi)$ (denoted by $z_{k,n}$) and $P_{n+1}^{(\lambda+1)}(x;\phi)$ (denoted by $y_{k,n+1}$), where $E_1$ and $E_2$ are defined in  \eqref{Factor1-responsible-Interlacing} and \eqref{Factor2-responsible-Interlacing}, respectively, and $\phi$ satisfies \eqref{condition-phi-positive-discriminant}. }
		\label{tab3:Interlacingzeros2_MeixnerPoll}
		\footnotesize 
		\setlength{\tabcolsep}{12pt} 
		\begin{tabular}{@{}c cc | cc@{}}
			\toprule
			
			\addlinespace[2pt]
			& \multicolumn{2}{c|}{$n=8, \lambda=1.5, \phi=\frac{4\pi}{5}$} & \multicolumn{2}{c}{$n=9, \lambda=6, \phi=\frac{\pi}{5}$} \\
			& \multicolumn{2}{c|}{$E_1 = -0.298549, E_2=12.686$} & \multicolumn{2}{c}{$E_1=-13.062, E_2=-0.701821$} \\
			\midrule
			$k$ & $z_{k,8}$ & $y_{k,9}$ & $z_{k,9}$ & $y_{k,10}$ \\
			\midrule
			
			1  & $-0.952961$ & $-1.10319$ & $-33.3638$  & $-38.5756$    \\
			2  & $0.439051$  & $0.623579$ & $-25.8717$  & $-30.6562$    \\
			3  & $1.91719$   & $2.34762$  & $-20.2451$  & $-24.6537$    \\
			4  & $3.85165$   & $4.38971$  & $-15.6627$  & $-19.7182$    \\
			5  & $6.37288$   & $6.88542$  & $-11.7964$  & $-15.5127$    \\
			6  & $9.62921$   & $9.9474$   & $-8.46192$  & $-11.8529$    \\
			7  & $13.8994$   & $13.736$   & $-5.52321$  & $-8.61351$    \\
			8  & $19.8989$   & $18.5489$  & $-2.83297$  & $-5.68527$    \\
			9  & $-$         & $25.1429$  & $-0.116578$ & $-2.9326$     \\
			10 & $-$         & $-$        & $-$         & $-0.0832518$\\
			\midrule
			& \multicolumn{2}{c|}{$(x-E_1)P_{n+1}^{(\lambda+1)}(x;\phi)\prec (x-E_2) P_{n}^{(\lambda)}(x;\phi)$} & \multicolumn{2}{c}{$(x-E_2)P_{n+1}^{(\lambda+1)}(x;\phi)\prec (x-E_1) P_{n}^{(\lambda)}(x;\phi)$} \\
			\bottomrule
		\end{tabular}
		
		\vspace{0.2cm}
		\begin{minipage}{0.9\textwidth}
			\footnotesize
			\textbf{Note:}  Observe that $y_{1,9}<z_{1,8}<E_1 < z_{2,8}<y_{2,9}$ and $y_{6,9}<E_2 < y_{7,9}$ for $n=8, \lambda=1.5, \phi=\frac{4\pi}{5}$; similarly, $y_{5,10}<E_1<y_{6,10}$ and $y_{9,10}<z_{8,9}<E_2 < z_{9,9}<y_{10,10}$ for $n=9, \lambda=6, \phi=\frac{\pi}{5}$.
		\end{minipage}
	\end{table}	
    
	It is worth noting that in the statement of Corollary \ref{MP-Interlacing-two-points}(c), we include the condition $y_{k,n+1} < z_{l,n} < E_1 < z_{l+1,n} < y_{k+1,n+1}$ to ensure the interlacing result $(x-E_1)P_{n+1}^{(\lambda+1)}(x;\phi)\prec (x-E_2)P_{n}^{(\lambda)}(x;\phi)$. However, our numerical investigations suggest that this condition may not be an independent requirement, but is likely an intrinsic property of the zeros of $P_{n}^{(\lambda)}(x;\phi)$ and $P_{n+1}^{(\lambda+1)}(x;\phi)$.  Numerical experiments across a wide range of parameters for $n$, $\lambda$ and $\phi$ suggest that in every case where $E_1$ and $E_2$ were chosen as $y_{k,n+1} < E_1 < y_{k+1,n+1},$	$y_{k',n+1} < E_2 < y_{k'+1,n+1}$, the zeros of $P_{n}^{(\lambda)}(x;\phi)$ ($z_{l,n}$ and $z_{l+1,n}$), are positioned exactly as needed to satisfy the condition in Corollary \ref{MP-Interlacing-two-points}(c). Essentially, the condition $y_{k,n+1} < z_{l,n} < E_1 < z_{l+1,n} < y_{k+1,n+1}$  seems to be a natural consequence of the polynomial structure rather than a restriction. To illustrate this, we have included Tables \ref{tab2:Interlacingzeros1_MeixnerPoll} and \ref{tab3:Interlacingzeros2_MeixnerPoll}.

Based on these observations, we propose the following conjecture. \vspace{0.3cm}
	
	\textbf{Conjecture 1.}  Let $\{y_{k,n+1}\}_{k=1}^{n+1}$ denote the zeros of $P_{n+1}^{(\lambda+1)}(x;\phi)$.  Let $E_1, E_2$ be as defined in \eqref{Factor1-responsible-Interlacing} and \eqref{Factor2-responsible-Interlacing}, respectively, for $\lambda >0$ and $\phi$ satisfying \eqref{condition-phi-positive-discriminant}. If there exist distinct $k, k' \in \{1, \dots, n\}$ such that $y_{k,n+1} < E_1 < y_{k+1,n+1}$ and $y_{k',n+1} < E_2 < y_{k'+1,n+1}$, then  $\displaystyle (x-E_1)P_{n+1}^{(\lambda+1)}(x;\phi) \prec (x-E_2)P_{n}^{(\lambda)}(x;\phi)$ or $\displaystyle (x-E_2)P_{n+1}^{(\lambda+1)}(x;\phi) \prec (x-E_1)P_{n}^{(\lambda)}(x;\phi)$ holds.
	
	\section{Pseudo-Jacobi polynomials}\label{Pseudo Jacobi polynomials}
	For real parameters $a$ and $b$, the monic Pseudo-Jacobi polynomials, denoted by $P_n(x; a, b)$, are expressed in terms of the hypergeometric function as \cite{Koekoek-Book-HypergeometricOP}:
	\begin{equation}
		P_n(x; a, b) = \frac{2^n(a + ib + 1)_n}{i^n(2a + n + 1)_n} \, {}_2F_1 \left( \left. \begin{matrix} -n, 2a + n + 1 \\ a + ib + 1 \end{matrix} \, \right| \, \frac{1 - ix}{2} \right),
	\end{equation}
	where $n$ is a nonnegative integer. These polynomials satisfy an orthogonality relation on the real line with respect to the measure $d\mu(x) = (1 + x^2)^a e^{2b \arctan x}dx$, provided that $b \in \mathbb{R}$ and $a < -n$ (cf. \cite{Jordaan-Tookos-JAT-2014}). 

	\begin{lemma}
		Let $n\in \mathbb{N}, 	a < -n - \frac{3}{2}$, $a\neq -n-2$, and $b \in \mathbb{R}$. The quadratic polynomial defined by
		\begin{equation}\label{Quadraticpolynomial-Pseudo-Jacobi-Interlacing}
			R_{n,a,b}(x) = x^2 - \frac{b(n+1)}{(a+n+1)(a+n+2)}x + \frac{(a+n+1)(a+n+2)(2a+n+2)+b^2(n+1)}{(a+n+1)(a+n+2)(2a+2n+3)}
		\end{equation}
		has two distinct real zeros if and only if		\begin{equation} \label{condition-a-realzeros}
			a < \frac{-3n - 5 - \sqrt{n^2+2n+5}}{4}~\text{with}~a\neq -n-2,
		\end{equation}
		and 
		\begin{equation} \label{condition-b-realzeros}
			|b| > \sqrt{ \frac{4(a+n+1)^2(a+n+2)^2(-2a-n-2)}{(n+1) \left| 4(a+n+1)^2 - (2n-2)(a+n+1) -(n+1) \right| } }.
		\end{equation}
	\end{lemma}

    \begin{proof} 
	To simplify the algebraic manipulation, we denote $	u=a+n+1.$
	Given the range $a<-n-\frac{3}{2}$, the range for $u$ becomes $u<-\frac{1}{2}$. Moreover, since $a\neq -n-2$, we have $u\neq -1$, so the polynomial is well-defined. Now, \eqref{Quadraticpolynomial-Pseudo-Jacobi-Interlacing} reads as
	\begin{align*}
		P(x)	=	x^2	-	\frac{b(n+1)}{u(u+1)}x	+	\frac{u(u+1)(2u-n)+b^2(n+1)}	{u(u+1)(2u+1)}.
	\end{align*}
	The zeros of the quadratic polynomial are real and distinct if and only if the discriminant $\Delta:=\operatorname{Disc}_x(P)$ is strictly positive. We have
	\begin{align}
		\Delta	&=	\left(\frac{b(n+1)}{u(u+1)}\right)^2	-	4\left(
	\frac{u(u+1)(2u-n)+b^2(n+1)}
		{u(u+1)(2u+1)}
		\right) \nonumber	\\&=	\frac{
			b^2(n+1)\left[-4u^2+(2n-2)u+(n+1)\right]-	4u^2(u+1)^2(2u-n)}
		{u^2(u+1)^2(2u+1)}\nonumber \\&
\label{eq:Delta-Ku}
	=\frac{	b^2(n+1)K(u)-4u^2(u+1)^2(2u-n)	}
		{u^2(u+1)^2(2u+1)}
	\end{align}
	where $	K(u):=-4u^2+(2n-2)u+(n+1)$.
	Since $u<-\frac{1}{2}$, the denominator $u^2(u+1)^2(2u+1)$ is negative. Therefore, $\Delta>0$ is equivalent to the condition that  the numerator in \eqref{eq:Delta-Ku} is negative, that is,
	\begin{align}	\label{eqnum_condition}
		b^2(n+1)K(u)-4u^2(u+1)^2(2u-n)
		< 0.
	\end{align}
    Since $u<-\frac{1}{2}$ implies $ 2u-n<0$, \eqref{eqnum_condition} forces $K(u)<0$.
		 We now characterize the condition $K(u)<0$ and rewrite \eqref{eqnum_condition} under this condition. First, the roots of $K(u)=0$ are
	\begin{align*}
		u	=	\frac{n-1\pm\sqrt{n^2+2n+5}}{4}.
	\end{align*}
	Since $K(u)$ is a downward-opening parabola, it is negative outside the interval determined by these two roots. In the range $u<-\frac{1}{2}$, apart from the already imposed condition $u\neq -1$, condition \eqref{condition-a-realzeros} is equivalent to
	\begin{align*}
		u	=	a+n+1	<	\frac{n-1-\sqrt{n^2+2n+5}}{4}.
	\end{align*}
	Hence, condition \eqref{condition-a-realzeros} is equivalent to $K(u)<0$ for $u<-\frac{1}{2}$.
	Therefore, $K(u)=-|K(u)|$.	Next, since $u=a+n+1$, we have
	\begin{align*}
		|K(u)|=	\left|	4(a+n+1)^2	-(2n-2)(a+n+1)-(n+1)\right|.
	\end{align*}
	Therefore, condition \eqref{condition-b-realzeros} can be written in terms of $u$ as
	\begin{align*}
	|b|	>\sqrt{	\frac{4u^2(u+1)^2(n-2u)}	{(n+1)|K(u)|}	}.
	\end{align*}
	Equivalently,
	\begin{align*}
		b^2(n+1)|K(u)|	>	4u^2(u+1)^2(n-2u).
	\end{align*}
	Since $K(u)=-|K(u)|$ and $n-2u=-(2u-n)$, this gives
	\begin{align*}
		-b^2(n+1)K(u)	>	-4u^2(u+1)^2(2u-n).
	\end{align*}
	Equivalently,
	\begin{align*}
		b^2(n+1)K(u)	-	4u^2(u+1)^2(2u-n)	< 0.
	\end{align*}
	Thus, $K(u)<0$ makes \eqref{condition-b-realzeros} and \eqref{eqnum_condition} equivalent. Therefore, \eqref{eqnum_condition}  holds if and only if conditions \eqref{condition-a-realzeros}  and \eqref{condition-b-realzeros} hold. Further, the same equivalence holds for $\Delta>0$. Hence, $R_{n,a,b}$ has distinct real zeros if and only if these conditions hold. This completes the proof. 
\end{proof}

	The quadratic polynomial $R_{n,a,b}(x)$ can be written in factorized form as
	\begin{equation}\label{Factorized-quadratic-PJP}
	R_{n,a,b}(x) = (x - E_{n,a,b})(x - F_{n,a,b}),
	\end{equation}
	where $E_1:=E_{n,a,b}$ and $E_2:=F_{n,a,b}$ are  real and given by
	\begin{equation}\label{PJP-Factor1-responsible-Interlacing}
	E_1 = \frac{b(n+1) - \sqrt{\frac{\mathcal{N}}{2a+2n+3}}}{2(a+n+1)(a+n+2)},
	\end{equation}
	and
	\begin{equation}\label{PJP-Factor2-responsible-Interlacing}
	E_2 = \frac{b(n+1) + \sqrt{\frac{\mathcal{N}}{2a+2n+3}}}{2(a+n+1)(a+n+2)}.
	\end{equation}
	Under the conditions $a<-n-2$ and \eqref{condition-b-realzeros},
the zeros are distinct and satisfy $E_1<E_2$.
	\noindent Here,  $\mathcal{N}$ is given by
	\begin{multline*}
	\mathcal{N} = b^2(n+1) \Big[ -4(a+n+1)^2 + (2n-2)(a+n+1) + (n+1) \Big] \\
	- 4(a+n+1)^2 (a+n+2)^2 (2a+n+2).
	\end{multline*}
	The monic Pseudo-Jacobi polynomials satisfy the following mixed recurrence relation, which is very useful for establishing interlacing properties between different sequences of these polynomials

	$\displaystyle \frac{2(2a+n+1)_2\left((a+n+1)^2+b^2\right)}{(n+a+1)(2a+2n+1)_3}P_n(x;a,b)$
	\begin{align}\label{PJ-MixedRR_nab_(n+1)(a+1)b_n(a+1)b}	=-\left(x-\frac{b}{n+a+1}\right)P_{n+1}(x;a+1,b)+(x-E_1)(x-E_2)P_{n}(x;a+1,b),
	\end{align}
	where $E_1$ and $E_2$ are given by \eqref{PJP-Factor1-responsible-Interlacing} and \eqref{PJP-Factor2-responsible-Interlacing}. This relation can be verified by comparing coefficients.
	
	\begin{corollary}\label{PJP-Interlacing-two-points}
	Let $n \in \mathbb{N}$, let $a <-n-2$ and suppose that $b$ satisfies \eqref{condition-b-realzeros}. Let $\{z_{k,n}\}_{k=1}^{n}$ and $\{y_{k,n+1}\}_{k=1}^{n+1}$ denote the zeros of the polynomials $P_n(x;a,b)$ and  $P_{n+1}(x;a+1,b)$, respectively.  Assume that $P_n(x;a,b)$ and $P_{n+1}(x;a+1,b)$ have no common zero. Then the configuration $E_1 < y_{1,n+1}$ and $E_2 > y_{n+1,n+1}$ is not possible,  where $E_1$ and $E_2$ are given in \eqref{PJP-Factor1-responsible-Interlacing} and \eqref{PJP-Factor2-responsible-Interlacing}. Also,  $P_{n+1}(x;a+1,b)\prec P_{n}(x;a,b)$ whenever $E_1$ and $E_2$ lie in the same interval defined by the sequence $\{y_{k,n+1}\}$; that is, when both lie in $(-\infty,y_{1,n+1})$, $(y_{n+1,n+1},\infty)$, or $(y_{k,n+1},y_{k+1,n+1})$ for some fixed $k$.
	
	Moreover, the following interlacing properties hold:
	\begin{enumerate}[label=\alph*)]
		\item If $E_1 < y_{1,n+1}$ and $y_{k,n+1} < E_2 < y_{k+1,n+1}$ for some $k \in \{1,\dots,n\}$, then  $(x-E_2)P_n(x;a,b)\prec P_{n+1}(x;a+1,b)$.
		
		\item If $E_2 > y_{n+1,n+1}$ and $y_{k,n+1} < E_1 < y_{k+1,n+1}$ for some $k \in \{1,\dots,n\}$, then $P_{n+1}(x;a+1,b) \prec (x-E_1)P_n(x;a,b)$.
		
		\item If there exist distinct $k,k' \in \{1,\dots,n\}$ such that 
		$y_{k,n+1} < E_1 < y_{k+1,n+1}$ and $y_{k',n+1} < E_2 < y_{k'+1,n+1},$
		then, for $l\in\{1,\dots,n-1\}$,
		\begin{itemize}
			\item [(i)] $(x-E_1)P_{n+1}(x;a+1,b)\prec (x-E_2)P_n(x;a,b)$ whenever $y_{k,n+1} <z_{l,n}< E_1<z_{l+1,n} < y_{k+1,n+1}$,
			\item[(ii)] $(x-E_2)P_{n+1}(x;a+1,b)\prec (x-E_1)P_n(x;a,b)$ whenever $y_{k',n+1} <z_{l,n}< E_2<z_{l+1,n} < y_{k'+1,n+1}$.
		\end{itemize}  
	\end{enumerate}
	\end{corollary}
	\begin{proof}
	We identify the polynomials in the mixed recurrence relation \eqref{PJ-MixedRR_nab_(n+1)(a+1)b_n(a+1)b} as $G_{n+1}(x) := P_{n+1}(x;a+1,b)$, $Q_{n}(x) := P_{n}(x;a+1,b)$, and $P_{n}(x) := P_{n}(x;a,b)$. We define the coefficient functions $A(x) := \frac{2(2a+n+1)_2\left((a+n+1)^2+b^2\right)}{(n+a+1)(2a+2n+1)_3}$ and $B(x) := -\left(x-\frac{b}{n+a+1}\right)$. Under the parameter constraints  given in \eqref{condition-a-realzeros}, we observe that $A(x) > 0$. Furthermore, the conditions \eqref{condition-a-realzeros} and \eqref{condition-b-realzeros} ensure that the roots $E_1$ and $E_2$ defined in \eqref{PJP-Factor1-responsible-Interlacing} and \eqref{PJP-Factor2-responsible-Interlacing} are real numbers. The function $B$ vanishes only at $x=\frac{b}{n+a+1}$. A direct calculation using \eqref{Quadraticpolynomial-Pseudo-Jacobi-Interlacing} gives \[ R_{n,a,b}\left(\frac{b}{n+a+1}\right) = \frac{(2a+n+2)\left((a+n+1)^2+b^2\right)} {(a+n+1)^2(2a+2n+3)} \neq 0, \]under the assumption $a<-n-2$ and hence $B(E_i)\neq 0$, $i=1,2$. Finally, we recall that for $n \in \mathbb{N}$, $a<-n-\frac32$, and $b\in \mathbb{R}$, the zeros of the polynomials $P_{n}(x;a+1,b)$ and $P_{n+1}(x;a+1,b)$  interlace. Consequently, all assumptions of Theorem~\ref{MainTheorem4} are satisfied, and the result follows immediately.
	\end{proof}

    Similar to Corollary \ref{MP-Interlacing-two-points}(c), the condition $y_{k,n+1} < z_{l,n} < E_1 < z_{l+1,n} < y_{k+1,n+1}$ in Corollary \ref{PJP-Interlacing-two-points}(c) appears to be an inherent structural property rather than an independent restriction. Numerical investigations, summarized in Tables \ref{tab4:Interlacingzeros1_PseudoJacobi} and \ref{tab5:Interlacingzeros2_PseudoJacobi}, confirm that whenever $E_1$ and $E_2$ (given in \eqref{PJP-Factor1-responsible-Interlacing} and \eqref{PJP-Factor2-responsible-Interlacing}, respectively) lie between consecutive zeros of $P_{n+1}(x;a+1,b)$, the zeros of $P_{n}(x;a,b)$ consistently position themselves to satisfy the required interlacing. This suggests that the condition is a natural consequence of the polynomial structure, leading to the following conjecture.
	\vspace{0.1cm}

	\textbf{Conjecture 2.}  Let $\{y_{k,n+1}\}_{k=1}^{n+1}$ be the zeros of $P_{n+1}(x;a+1,b)$. Let $E_1$ and $E_2$ be given by \eqref{PJP-Factor1-responsible-Interlacing} and \eqref{PJP-Factor2-responsible-Interlacing}, where $a$ and $b$ satisfy the corresponding conditions \eqref{condition-a-realzeros} and \eqref{condition-b-realzeros}. If there exist distinct $k, k' \in \{1, \dots, n\}$ such that $y_{k,n+1} < E_1 < y_{k+1,n+1}$ and $y_{k',n+1} < E_2 < y_{k'+1,n+1}$,  then  $\displaystyle (x-E_1) P_{n+1}(x;a+1,b) \prec (x-E_2)P_{n}(x;a,b)$ or $\displaystyle (x-E_2)P_{n+1}(x;a+1,b) \prec (x-E_1)P_{n}(x;a,b)$ holds.
	
In this manuscript, numerical computations of the zeros were performed using $\text{Mathematica}^{\text{\textregistered}}$ software.

	\begin{table}[H]
	\centering
	\caption{Zeros $z_{k,n}$ and $y_{k,n+1}$ of the polynomials $P_{n}(x;a,b)$ and $P_{n+1}(x;a+1,b)$, where $a$ and $b$ satisfy conditions \eqref{condition-a-realzeros} and \eqref{condition-b-realzeros}. The parameters $E_1$ and $E_2$ are defined in \eqref{PJP-Factor1-responsible-Interlacing} and \eqref{PJP-Factor2-responsible-Interlacing}, respectively.}
	\label{tab4:Interlacingzeros1_PseudoJacobi}
	\footnotesize 
	\setlength{\tabcolsep}{12pt} 
	\begin{tabular}{@{}c cc | cc@{}}
		\toprule
		
		\addlinespace[2pt]
		& \multicolumn{2}{c|}{$n=6, a=-9.60, b=2.81$} & \multicolumn{2}{c}{$n=7, a=-12.83, b=-5.85$} \\
		& \multicolumn{2}{c|}{$E_1 = -0.103838, E_2=4.84379$} & \multicolumn{2}{c}{$E_1=-2.44111, E_2=-0.087948$} \\
		\midrule
		$k$ & $z_{k,6}$ & $y_{k,7}$ & $z_{k,7}$ & $y_{k,8}$ \\
		\midrule
		
		1 & $-0.415663$ & $-0.442364$ & $-2.89991$ & $-5.15497$ \\
		2 & $-0.000398727$ & $0.00700039$ & $-1.76843$ & $-2.84266$ \\
		3 & $0.356554$ & $0.389392$ & $-1.16387$ & $-1.80435$ \\
		4 & $0.767051$ & $0.836066$ & $-0.757969$ & $-1.19766$ \\
		5 & $1.37113$ & $1.50827$ & $-0.440675$ & $-0.780726$ \\
		6 & $2.6032$ & $2.85726$ & $-0.154304$ & $-0.453039$ \\
		7 & $-$ & $7.15039$ & $0.159308$ & $-0.156763$ \\
		8 & $-$ & $-$ & $-$ & $0.169808$ \\
		\midrule
		& \multicolumn{2}{c|}{$(x-E_1)P_{7}(x;a+1,b)\prec (x-E_2) P_{6}(x;a,b)$} & \multicolumn{2}{c}{$(x-E_2)P_{8}(x;a+1,b)\prec (x-E_1) P_{7}(x;a,b)$} \\
		\bottomrule
		
	\end{tabular}
	
	\vspace{0.2cm}
	\begin{minipage}{0.9\textwidth}
		\footnotesize
		\textbf{Note:} Observe that $y_{1,7}<z_{1,6}<E_1 < z_{2,6}<y_{2,7}$ and $y_{6,7}<E_2 < y_{7,7}$ for $n=6, a=-9.60, b=2.81$; similarly, $y_{2,8}<E_1<y_{3,8}$ and $y_{7,8}<z_{6,7}<E_2 < z_{7,7}<y_{8,8}$ for $n=7, a=-12.83, b=-5.85$.
	\end{minipage}
	\end{table}

	\begin{table}[H]
	\centering
	\caption{Zeros $z_{k,n}$ and $y_{k,n+1}$ of the polynomials $P_{n}(x;a,b)$ and $P_{n+1}(x;a+1,b)$, where $a$ and $b$ satisfy conditions \eqref{condition-a-realzeros} and \eqref{condition-b-realzeros}. The parameters $E_1$ and $E_2$ are defined in \eqref{PJP-Factor1-responsible-Interlacing} and \eqref{PJP-Factor2-responsible-Interlacing}, respectively.}
	\label{tab5:Interlacingzeros2_PseudoJacobi}
	\footnotesize 
	\setlength{\tabcolsep}{12pt} 
	\begin{tabular}{@{}c cc | cc@{}}
		\toprule
		
		\addlinespace[2pt]
		& \multicolumn{2}{c|}{$n=6, a=-9.60, b=-2.70$} & \multicolumn{2}{c}{$n=7, a=-9.23, b=-0.22$} \\
		& \multicolumn{2}{c|}{$E_1 = -4.62523, E_2=0.0593014$} & \multicolumn{2}{c}{$E_1=-5.17479, E_2=-1.06178$} \\
		\midrule
		$k$ & $z_{k,6}$ & $y_{k,7}$ & $z_{k,7}$ & $y_{k,8}$ \\
		\midrule
		
		1 & $-2.54341$ & $-6.92597$ & $-1.91636$ & $-8.2896$ \\
		2 & $-1.33862$ & $-2.7751$ & $-0.922535$ & $-1.99692$ \\
		3 & $-0.744617$ & $-1.46463$ & $-0.4255$ & $-0.911932$ \\
		4 & $-0.338219$ & $-0.807117$ & $-0.0580861$ & $-0.381732$ \\
		5 & $0.0180625$ & $-0.366362$ & $0.297398$ & $0.0146716$ \\
		6 & $0.436834$ & $0.0150958$ & $0.746242$ & $0.420477$ \\
		7 & $-$ & $0.469881$ & $1.57386$ & $1.00162$ \\
		8 & $-$ & $-$ & $-$ & $2.44817$ \\
		\midrule
		& \multicolumn{2}{c|}{$(x-E_2)P_{7}(x;a+1,b)\prec (x-E_1) P_{6}(x;a,b)$} & \multicolumn{2}{c}{$(x-E_2)P_{8}(x;a+1,b)\prec (x-E_1) P_{7}(x;a,b)$} \\
		\bottomrule
		
	\end{tabular}
	
	\vspace{0.2cm}
	\begin{minipage}{0.9\textwidth}
		\footnotesize
		\textbf{Note:} Observe that $y_{1,7}<E_1 < y_{2,7}$ and $y_{6,7}<z_{5,6}<E_2< z_{6,6}<y_{7,7}$ for $n=6, a=-9.60, b=-2.70$; similarly, $y_{1,8}<E_1<y_{2,8}$ and $y_{2,8}<z_{1,7}<E_2 < z_{2,7}<y_{3,8}$ for $n=7, a=-9.23, b=-0.22$.
	\end{minipage}
	\end{table}

	\section{Conclusion}
		This study establishes a systematic framework for completing the interlacing of zeros in polynomial sequences where interlacing fails by exactly two points. We have identified a specific degree-two polynomial whose zeros provide the necessary points to restore a complete interlacing structure. We have also analyzed how the positions of these points relative to the zeros of the higher-degree polynomial determine the resulting interlacing configurations.

        We have also shown that, once a completed interlacing relation is obtained by
		inserting the two explicit points, the resulting zero configuration provides
		the sign information needed to derive Wronskian-type bounds for the original
		polynomial pair. This gives a route from completed interlacing to Wronskian-type bounds, without relying essentially on orthogonality at the level of the general result. 
	
	We have shown the utility of this approach by applying it to several orthogonal families. For Jacobi polynomials, the method refines earlier results by explicitly identifying the two completion points for the interlacing of $P_n^{(\alpha,\beta)}$ and $P_{n+1}^{(\alpha+1,\beta+1)}$. The Jacobi analysis also shows that common zeros of the shifted pair are a non-generic phenomenon, i.e., for fixed degree, they occur only on a thin exceptional subset of the parameter domain. For Jacobi polynomials with the shift $(\alpha,\beta,n)\mapsto(\alpha+1,\beta+1,n+1)$, the Wronskian-type bounds specialize to cross-family Tur\'an-type inequalities for the doubly shifted pair. The associated Jacobi quotient also admits a  Stieltjes-transform representation through a positive partial-fraction expansion, which yields the corresponding complete-monotonicity consequences. The completed-interlacing framework also addresses an open question for Meixner-Pollaczek polynomials of consecutive degrees with the parameter shifted by one, and yields analogous completed-interlacing results for Pseudo-Jacobi polynomials.
		
	While this work focuses on identifying the degree-two polynomial using the general mixed recurrence relation given in \eqref{MoreGeneralMixedTTRR}, the resulting completed interlacing also leads to various analytic consequences. Thus, the unified approach presented here provides a practical tool for investigating similar interlacing gaps in both orthogonal and non-orthogonal polynomial families.
	
	\bibliographystyle{cas-model2-names}
	\bibliography{references-Twopoints}
\end{document}